\documentclass[a4paper]{article}
\usepackage{graphicx}
\usepackage{amsfonts}
\usepackage{amsmath}
\usepackage{amssymb}
\usepackage{fancyhdr}
\usepackage{titlesec}
\usepackage{indentfirst}
\usepackage{booktabs}
\usepackage{verbatim}
\usepackage{color}
\usepackage{amsthm}
\usepackage{subfigure}
\usepackage{abstract}
\usepackage{fullpage}
\usepackage{hyperref}
\usepackage[margin=0pt]{geometry}
\usepackage{lipsum}
\usepackage{layout}
\usepackage{mathtools}
\usepackage[page,header]{appendix}
\usepackage{titletoc}
\numberwithin{equation}{section}
\newtheorem{theorem}{Theorem}[section]
\newtheorem{lemma}[theorem]{Lemma}
\newtheorem{corollary}[theorem]{Corollary}

\newtheorem{remark}[theorem]{Remark}

\DeclareMathOperator{\disc}{disc}
\DeclareMathOperator{\Gal}{Gal}
\DeclareMathOperator{\ind}{ind}
\DeclareMathOperator{\expo}{exp}
\DeclareMathOperator{\Disc}{Disc}
\DeclareMathOperator{\Nm}{Nm}

\newcommand{\rd}{\,\mathrm{d}}

\begin{document}

\title{Malle's Conjecture for $S_n\times A$ for $n = 3,4,5$}
\author{Jiuya Wang}
\newcommand{\Addresses}{{
		\bigskip
		\footnotesize		
		Jiuya Wang, \textsc{Department of Mathematics, University of Wisconsin-Madison, 480 Lincoln Dr., Madison, WI 53706, USA
		}\par\nopagebreak
		\textit{E-mail address}: \texttt{jiuyawang@math.wisc.edu}
	}}
\maketitle	
	\begin{abstract}
		We propose a framework to prove Malle's conjecture for the compositum of two number fields based on proven results of Malle's conjecture and good uniformity estimates. Using this method we can prove Malle's conjecture for $S_n\times A$ over any number field $k$ for $n=3$ with $A$ an abelian group of order relatively prime to 2, for $n= 4$ with $A$ an abelian group of order relatively prime to 6 and for $n=5$ with $A$ an abelian group of order relatively prime to 30. As a consequence, we prove that Malle's conjecture is true for $C_3\wr C_2$ in its $S_9$ representation, whereas its $S_6$ representation is the first counter example of Malle's conjecture given by Kl\"uners. 
		
	\end{abstract}
	
\pagenumbering{arabic}	
\section{Introduction}
There are only finitely many number fields with bounded discriminant, therefore it makes sense to ask how many there are. Malle's conjecture aims to answer the asymptotic question for number fields with prescribed Galois group. Let $k$ be a number field and $K/k$ be a degree $n$ extension with Galois closure $\tilde{K}/k$, we define $\Gal(K/k)$ to be $\Gal(\tilde{K}/k)$ as a transitive permutation subgroup of $S_n$ where the permutation action is defined by its action on the $n$ embeddings of $K$ into $\bar{k}$. Let $N_k(G, X)$ be the number of isomorphism classes of extensions of $k$ with Galois group isomorphic to $G$ as a permutation subgroup of $S_n$ and absolute discriminant bounded by $X$.  Malle's conjecture states that $N_k(G,X) \sim C X^{1/a(G)} \ln^{b(k,G)-1} X$ where $a(G)$ depends on the permutation representation of $G$ and $b(k,G)$ depends on both the permutation representation and the base field $k$. See section $2.3$ for explanations on the constants. 

Malle's conjecture has been proven for abelian extensions over $\mathbb{Q}$ \cite{M85} and over arbitrary bases \cite{Wri89}. However, for non-abelian groups, there are only a few cases known. The first case is $S_3$ cubic fields proved by Davenport and Heilbronn \cite{DH71} over $\mathbb{Q}$ and later proved by Datskovsky and Wright \cite{DW88} over any $k$. Bhargava and Wood \cite{BW08} and Belabas and Fouvry \cite{BF} independently proved the conjecture for $S_3$ sextic fields. The cases of $S_4$ quartic fields \cite{Bha05} and $S_5$ quintic fields \cite{Bha10} over $\mathbb{Q}$ are also proved by Bhargava. In \cite{BSW15}, these cases are generalized to arbitrary $k$ by Bhargava, Shankar and Wang. The case of $D_4$ quartic fields over $\mathbb{Q}$ is proved by Cohen, Diaz y Diaz and Olivier \cite{CDyDO02}.

The main result of this paper is to prove Malle's conjecture for $S_n\times A$ in its $S_{n|A|}$ representation for $n=3,4,5$ with certain families of $A$. 

\begin{theorem}\label{Thm1}
	Let $A$ be an abelian group and let $k$ be any number field. Then there exists $C$ such that the asymptotic distribution of $S_n\times A$-number fields over $k$ by absolute discriminant is $$N_k(S_n\times A, X) \sim C X^{1/|A|}$$
	in the following cases:
	\begin{enumerate}
	\item $n=3$, if $2\nmid |A|$;
	 \item $n=4$, if $2,3 \nmid |A|$;
	 \item $n=5$, if $2,3,5\nmid |A|$.
	 \end{enumerate}
\end{theorem}
%
Please see section 2.3 for the explanation that this agrees with Malle's conjecture. We can write out the constant $C$ explicitly given the generating series of $A$-extensions by discriminant, see e.g.\cite{M85,Woo10a, Wri89}. The constant $C$ could be written as a finite sum of Euler products when the generating series of $A$-extensions is a finite sum of Euler products. 

For example, if we count all homomorphisms $G_{\mathbb{Q}}\to S_3\times C_3$ that surject onto the $S_3$ factor, the asymptotic count of these homomorphisms by discriminant is
\begin{equation}
\begin{aligned}	
2\prod_{p}c_p X^{1/3}，
\end{aligned}
\end{equation}
where $c_p = (1+ p^{-1}+ 5p^{-2}+ 2p^{-7/3})(1-p^{-1})$ for $p\equiv 1 \mod 3$ and $c_p = (1+p^{-1}+ p^{-2})(1-p^{-1})$ for $p\equiv 2\mod 3$. For $p=3$, we use the database \cite{LMFDB} to compute that $c_3 = 3058\cdot 3^{-5}+ 4\cdot 3^{4/3}\approx 29.8914$. If we count the actual number of isomorphism classes of $S_3\times C_3$ extensions, i.e., all surjections $G_{\mathbb{Q}} \to S_3\times C_3$ up to an automorphism, the asymptotic constant is naturally a difference of two Euler products. One is given above divided by $|\text{Aut}(S_3\times C_3)| = 12$ and the other one comes from the subtraction of the $S_3$ extensions.

However, Malle's conjecture has been shown to be not generally correct. Kl\"uners \cite{Klu} shows that the conjecture does not hold for $C_3\wr C_2$ number fields over $\mathbb{Q}$ in its $S_6$ representation, where Malle's conjecture predicts a smaller power for $\ln X$ in the main term. See \cite{Klu} and \cite{Tur08} for suggestions on how to fix the conjecture. And by relaxing the precise description of the power for $\ln X$, weak Malle's conjecture states that $N_k(G,X) \sim C X^{1/a(G)+\epsilon}$. Kl\"uners and Malle proved weak Malle's conjecture for nilpotent groups \cite{Klu01}. Kl\"uners also proved the weak conjecture for groups in the form of $C_2\wr H$ \cite{Klu12} under mild conditions on $H$. 

Notice that for Kl\"uners' counter example, $C_3\wr C_2\simeq S_3\times C_3$, we have the following corollary. 
\begin{corollary}
	Malle's conjecture holds for $C_3\wr C_2$ in its $S_9$ representation over any number field $k$.
\end{corollary}
Counting non-Galois number fields could be considered as counting Galois number fields by discriminant of certain subfields. A natural question thus will be: what kind of subfields provide the discriminant as an invariant by which the asymptotic estimate is as predicted by Malle. 

Malle considered the compatibility of the conjecture under taking compositum in his original paper \cite{Mal02} and estimates both the lower bound and upper bound of asymptotic distribution for compositum when the two Galois groups have no common quotient. By working out a product argument, we show a better lower bound in general, see Corollary 3.3. And by analyzing the behavior of the discriminant carefully and applying good uniformity results, we show a better upper bound for our cases $S_n\times A$, see Theorem \ref{Thm1}, which gives the same order of main term and actually matches Malle's prediction. 

In section 2, we analyze the discriminant of a compositum in terms of each individual discriminant, and then compute the case explicitly for $S_n\times A$. Then we check that our computation agrees with Malle's prediction. In section 3, we prove the product argument in two different cases. In section $4$, we include and prove some necessary uniformity results for $S_n$ extensions where $n=3,4,5$ and abelian extensions.  Finally, in section $5$ we prove our main theorems based on what we have developed before. \\

\textbf{Notations}\\
$p$: a finite place in base field $k$\\
$|\cdot|$:  absolute norm $\Nm_{k/\mathbb{Q}}$\\
$\disc(K/k)$ : relative discriminant ideal in base field $k$\\
$\disc_p(K/k)$: $p$-part of $\disc(K/k)$\\
$\Disc(K)$:  absolute norm of $\disc(K/k)$ to $\mathbb{Q}$\\
$\Disc_p(K)$: absolute norm of $\disc_p(K/k)$\\
$\tilde{K}$: Galois closure of $K$ over base field $k$\\
$\ind(\cdot)$: the index $n$ - $\sharp \{\text{orbits}\}$ for a cycle or minimum value of index among non-identity elements for a permutation group\\
$N_k(G,X)$: the number of isomorphic classes of $G$ extension over $k$ with $\Disc$ bounded by $X$\\
$f(x)\sim g(x)$: $\lim_{x\to\infty} \frac{f(x)}{g(x)} = 1$

\section{Discriminant of Compositum}
\subsection{General Description}
We will describe the relation between $\Disc(KL)$ and $\Disc(K)$, $\Disc(L)$ when $\tilde{K}$ and $\tilde{L}$ have trivial intersection. 
\begin{theorem}\label{dpro}
	 Let $K/k$ and $L/k$ be extensions over $k$ which intersect trivially, then $\Disc(KL) \le \Disc(K)^n \Disc(L)^m$, where $n = [L: k]$, $m = [K: k]$. 
\end{theorem}
\begin{proof}
If $k = \mathbb{Q}$, then the ring of integers $O_K$ and $O_L$ are free $\mathbb{Z}$-modules with rank $m$ and $n$. Then $\Disc(O_KO_L) = \Disc(K)^n \Disc(L)^m$ and $O_KO_L\subset O_{KL}$. Over arbitrary $k$, we have $\disc(S^{-1}O_K/S^{-1}O_k) = S^{-1}\disc(O_K/O_k)$ as an $O_k$-module, see e.g. Theorem $2.9$ \cite{Neu99}. We take $S=O_k\backslash p$ for some prime ideal $p\subset O_k$ to look at $\disc_p(K/k)$. Now $S^{-1}O_k\subset k$ is a discrete valuation ring with the unique maximal ideal $S^{-1}p$, and $S^{-1}O_K$ is a finitely generated $S^{-1}O_k$-module, therefore admits an integral basis. Notice that $S^{-1}(O_K)$ intersects trivially with $S^{-1}O_L$, so it follows $\disc_p(KL) \le \disc_p(K)^n\disc_p(L)^m$ similarly.
\end{proof}
This gives an upper bound of $\Disc(KL)$. To be more precise, we focus on the study of $\Disc(KL)$ at tamely ramified primes over arbitrary number field $k$. Firstly, any tame inertia group is cyclic, therefore it could be described by the generator. Secondly, suppose $I =\large \langle g\large \rangle$ at a certain finite place $p$, then the index of $g\in G\subset S_n$, $$\ind(g) = n - \sharp\{ \text{orbits}\} = \sum (e_i-1) f_i,$$ 
is exactly the exponent for the $p$-part of the relative discriminant ideal. So we can determine the discriminant at $p$ by looking at the cycle type of $g$. 


If $\tilde{K}\cap \tilde{L}  = k$, then $\Gal(\tilde{K}\tilde{L}/k) \simeq \Gal(\tilde{K}/k) \times \Gal(\tilde{L}/k)$, where the isomorphism is a product of the restrictions to $\tilde{K}$ and $\tilde{L}$. Say $\Gal(\tilde{K}/\mathbb{Q}) = G_1\subset S_m$ and $\Gal(\tilde{L}/\mathbb{Q}) = G_2\subset S_n$, then $G = G_1\times G_2$ has a natural permutation representation in $S_{mn}$. Suppose $\tilde{K}$ and $\tilde{L}$ are both tamely ramified at $p$ with $I_i = \large \langle g_i\large \rangle\subset G_i$ , for $i = 1,2$, then $\tilde{K}\tilde{L}$ is also tamely ramified since tamely ramified extensions are closed under taking compositum. And the inertia group is $I = \large \langle g\large \rangle = \large \langle (g_1, g_2)\large \rangle$ for $\tilde{K}\tilde{L}$ because the inertia group for a sub-extension behaves naturally as quotient. 

\begin{theorem}\label{indrp}
	Let $K$ and $L$ be given above, and let $e_i$, for $i = 1,2$, be the ramification indices of $\tilde{K}$ and $\tilde{L}$ at a tamely ramified $p$. If $(e_1, e_2) = 1$, then $\ind(g) = \ind(g_1)\cdot n + \ind(g_2)\cdot m - \ind(g_1)\cdot \ind(g_2)$. 
\end{theorem}
\begin{proof}
	Suppose $g_1\in G_1\subset S_m$ is a product of disjoint cycles $\prod c_k$, then $e_1$ will be the least common multiple of $|c_k|$, the length of cycles $c_k$ for all $k$. Similarly for $g_2$ as a product of cycles $\prod d_l$. Now embed $(g_1,g_2)$ to $S_{mn}$, the permutation action is naturally defined to be mapping $a_{i,j}$ to $a_{g_1(i), g_2(j)}$ for $1\le i\le m$, $1\le j\le n$. If $(e_1,e_2) = 1$, then for any $k,l$, $(|c_k|,|d_l|) = 1$ and $(c_k,d_l)$ forms a single cycle of length $|c_k||d_l|$ in $S_{mn}$.  So the number of orbits in $g$ is the product of number of orbits in $g_i$. Therefore $\ind(g) = mn - (m-\ind(g_1))(n-\ind(g_2)) = \ind(g_1)\cdot n + \ind(g_2)\cdot m - \ind(g_1)\cdot \ind(g_2)$.
\end{proof}
This gives a nice description of $\disc_p(KL)$ independent of the cycle type when the ramification indices are relatively prime. In general, to know $\ind(g)$ requires more information on the cycle type of $g_i$. 
\begin{theorem}\label{indge}
Let $K$ and $L$ be as given above, $g_1$ be a product of disjoint cycles $\prod c_k$ and $g_2$ be a product of disjoint cycles $\prod d_l$ where $g_i$ is the generator for a tame ramified $p$ for $\tilde{K}$and $\tilde{L}$, then $\ind(g) = mn - \sum_{k,l} \gcd(|c_k|, |d_l|)$. 
\end{theorem}
\begin{proof}
	Notice that we can write $\ind(g_1) = \sum_k (|c_k|-1)$. In general, $(c_k, d_l)$ is no longer a single orbit in $S_{mn}$. Instead, it splits into $\gcd(|c_k|, |d_l|)$ many orbits. So the summation is $\ind(g) = \sum_{k,l} (|c_k||d_l| - \gcd(|c_k|, |d_l|)) = mn - \sum_{k,l} \gcd(|c_k|, |d_l|)$.
\end{proof}	

\subsection{Discriminant for $S_n\times A$}
We will describe the example of $S_n\times A$ for our interests in detail here. We will only consider the cases where $n=3,4,5$ and $A$ is an odd order abelian group. 

Firstly, we take the example of $S_3\times A$ where $A = C_{l^k}$ is cyclic with odd prime power order $l^k$. Possible tame inertia generators in $S_3$ could be $(12)$, $(123)$. For $A\subset S_{|A|}$, possible generators are of the form $g = (123...l^k)$ or powers of $g$, i.e., a single cycle of length $l^k$ or a product of $l^r$ cycles of length $l^{k-r}$. So $\ind(g)$ is minimized when $g$ is $l^{k-1}$ product of cycles of length $l$, therefore $\ind(A)$ is $l^k-l^{k-1}$, and $\frac{|A|}{\ind(A)} = \frac{l}{l-1}$. If $l\ne 3$, then we can apply Theorem \ref{indrp} to get Table 1. The numbers in the table give the exponent for $p$ in $\disc_p$ for each field. 
\begin{table}[!htbp]
	\centering
	\begin{tabular}{|c|c|c|}
		\hline
		$S_3$ & $C_{l^k}$ & $S_3\times C_{l^k}$ \\
		\hline
		(12) & $l^k-l^r$ & $3l^k-2l^r$ \\
		\hline
		(123)& $l^k-l^r$ & $3l^k-l^r$ \\
		\hline
	\end{tabular}
	\caption{Table of $\Disc_p$ for $S_3\times C_{l^k}$, $l\ne3$ }
\end{table}\\
If $l = 3$, we apply Theorem \ref{indge} to get Table 2.
\begin{table}[!htbp]
	\centering
	\begin{tabular}{|c|c|c|}
		\hline
		$S_3$ & $C_{l^k}$ & $S_3\times C_{l^k}$ \\
		\hline
		(12) & $l^k-l^r$ & $3l^k-2l^r$ \\
		\hline
		(123) & $l^k-l^r$ & $3l^k-3l^r$ \\
		\hline
	\end{tabular}
	\caption{Table of $\Disc_p$ for $S_3\times C_{l^k}$, $l=3$ }
\end{table} \\
We do not include in the table the cases where one of the inertia groups is trivial since $\disc_p(KL) = \disc_p(K)^n\disc_p(L)^m$ at these $p$ from previous computation. To compute the precise table for general $A$, we can compute the table for all abelian $l$-groups and then apply Theorem \ref{indrp} inductively to combine different $l$-parts. The general pattern we need for the proof of the main theorems is:
\begin{lemma}\label{delta3}
	Let $A$ be an abelian group of odd order $m$ and $(12)$, $(123)$ be elements in $S_3$. Then for all $c\in A$, $\ind((12),c)/m > 2$, $\ind((123),c)/m>1$.
\end{lemma}
\begin{proof}
	For any abelian group $A$, $\frac{|A|}{\ind(A)}= \frac{p}{p-1}$ where $p$ is the minimal prime divisor of $|A|$, and $\frac{p}{p-1}<2$ if $p\ne 2$. This can be seen by combining the different $l$-parts of $A$ inductively. The value $\ind((12),c) = m+3\cdot \ind(c) - \ind(c) = m+2\cdot \ind(c)\ge m+ 2 \cdot \ind(A) > 2m$ because $\frac{|A|}{\ind(A)}<2$. 
	
	For $\ind((123),c)$, if $3\nmid |A|$, then $\ind((123), c) = 2m + 3\cdot \ind(c)-2\cdot \ind(c) = 2m+\ind(c)>m$ with no problem. If $3| |A|$, we separate $3$-part of $A$ to compute $\ind((123), c)$. Let $A = A_3\times A_{>3}$ where $A_3$ is the $3$-part of $A$ and $A_{>3}$ contains all $p>3$ part. Let $c = (c_3, c_{>3})$ be any element in $A$, then $\ind((123),c)= \ind((123),c_3,c_{>3}) = \ind(((123),c_3),c_{>3})$ where $((123),c_3)$ is an element in $S_3\times A_3$. Say $\ind((123),c_3) = i$, then
	\begin{equation}\label{2.1}
	\begin{aligned}	
    \ind((123), c_3, c_{>3}) &= i|A_{>3}|+ (3|A_3| - i)\cdot \ind(c_{>3}) \\
    &= i(|A_{>3}| - \ind(c_{>3})) + 3|A_3|\cdot \ind(c_{>3}).
	\end{aligned}
	\end{equation}
	Therefore the minimal value of $\ind((123),c)$ is obtained when both $i$ and $\ind(c_{>3})$ are smallest possible. The smallest possible $\ind(c_{>3})$ is $\ind(A_{>3})$. The smallest $\ind((123), c_3)$ is $\ind((123), e) = 2|A_3|$. Therefore, if $A = A_3$, then $2|A_3|/m = 2>1$. If $A_{>3}$ is non-trivial, then by (\ref{2.1}),  $\ind((123),c) \ge 2m+|A_3|\cdot \ind(A_{>3}) > m$.
\end{proof}

\begin{lemma}\label{delta4}
	Let $A$ be an abelian group of odd order and $2,3\nmid |A|=m$ and $(12)$, $(123)$, $(1234)$, $(12)(34)$ be elements in $S_4$. Then for all $c\in A$, $\ind((12), c)/m > 2$, $\ind((12)(34), c)/m>1$, $\ind((123), c)/m>3$, $\ind((1234), c)/m>2$. 
\end{lemma}
\begin{proof}
	We can apply Theorem \ref{indrp} since $2,3\nmid m$. Then $\ind((12), c) = m + 3\cdot \ind(c) \ge m+ 3\cdot \ind(A) > 2m$, $\ind((12)(34), c) = 2m+2\cdot \ind(c) > m$, $\ind((1234), c) = 3m+ \ind(c)>2m$, $\ind((123),c) = 2m+2\cdot \ind(c) \ge 2m+ 2\cdot \ind(A) \ge 2m+ 2\cdot \frac{4}{5}m > 3m$.
\end{proof}

\begin{lemma}\label{delta5}
	Let $A$ be an odd abelian group and $2,3, 5\nmid |A|=m$. Then $\forall c\in A$ and $k\in S_5$ , $\ind(k, c) /m\ge 1 + \ind(k)-1/7$. 
\end{lemma}
\begin{proof}
	We can apply Theorem \ref{indrp} since $2,3\nmid m$. Then $\ind(k, c) = m\ind(k)+ 5\ind(c) - \ind(k)\ind(c) =m\ind(k) + (5-\ind(k))\ind(c) = (m-\ind(c))\ind(k) + 5\ind(c)$. So for a certain $k$, the value is smallest when $\ind(c) = \ind(A) $. And at this time $\ind(k,c)/m = \ind(k) + (5-\ind(k))\frac{\ind(A)}{m} = \ind(k) + (5-\ind(k))\frac{p-1}{p}$ where $p$ is the smallest divisor of $m$ and $p\ge 7$. So $\ind(k)/m - \ind(k) =  (5-\ind(k))\frac{p-1}{p}\ge (5 - 4) \frac{6}{7} = \frac{1}{7}$.
\end{proof}

\subsection{Malle's Prediction for $S_n\times A$}
In this section we compute the value of $a(G)$ and $b(k, G)$ for $S_n\times A$. A similar discussion on $a(G)$ for a direct product of two Galois groups in general is in \cite{Mal02}. We include here for the convenience of the reader. Recall that given $G\subset S_n$ a permutation group, for each element $g\in G$, $\ind(g)= n - \sharp\{ \text{orbits of g}\}$. We define $a(G)$ to be the minimum value of $\ind(g)$ among all $g\neq e$. The absolute Galois group $G_k$ acts on the conjugacy classes of $G$ via its action on the character table of $G$. We define $b(k,G)$ to be the number of orbits within all conjugacy classes with minimal index. 

Let $G_i\subset S_{n_i}, i = 1,2$ be two permutation groups. Consider $G = G_1\times G_2\subset S_{n_1n_2}$. Suppose that $g_i\in G_i$ gives minimal index, then for $G\subset S_{n_1n_2}$, the minimal index will either come from $g_1\times e$ or $e\times g_2$ since for any $g\in G_2$, $\ind(g_1, e)\le \ind(g_1, g)$. One can compute $\ind(g_1\times e) = n_2\ind(g_1)$. Therefore $a(G) = \min\{ n_2\cdot a(G_1), n_1\cdot a(G_2)\} = n_1n_2\min\{ \frac{a(G_1)}{n_1}, \frac{a(G_2)}{n_2}  \}$. 

If $ \frac{a(G_1)}{n_1} < \frac{a(G_2)}{n_2}$, then $g\times e$ for all $g$ with $\ind(g) = a(G_1)$ are exactly the elements with minimal index in $G$. Irreducible representations of $G_1\times G_2$ are $\rho_1\otimes \rho_2$ where $\rho_i$ are irreducible representations of $G_i$ with character $\chi_i$. The corresponding character is $\chi_1\cdot\chi_2$. Therefore the $G_k$ action on $g\times e$ has the same orbit as its action on $g$. So $b(k, G) = b(k, G_1)$. 

Our case $S_n\times A$ satisfies the above condition, therefore $a (S_n\times A) = nm\min\{ \frac{1}{n}, \frac{p-1}{p}\} = m$ where $p$ is the smallest prime divisor of $|A| = m$ and $n=3, 4, 5$. And $b(k, S_n\times A) = b(k, S_n) = 1$. 

\section{Product Lemma}
This section answers the question: given two distributions $F_i$, $i = 1, 2$,  each describes the asymptotic distribution of some multi-set of positive integers $S_i$, i.e., $F_i(X) = \sharp\{ s\in S_i \mid s\le X\}$, what is the product distribution $P_{a,b}(X) = \sharp\{ (s_1,s_2)\mid s_i\in S_i, s_1^as_2^b\le X\} $ where $a, b>0$. We will split the discussion into two cases.
\begin{lemma}\label{proeq}
	Let $F_i(X)$, $i=1,2$, be as given above, $F_i(X) \sim A_i X^{n_i}\ln^{r_i} X$ where $0 <n_i\le 1$ and $r_i\in \mathbb{Z}_{\ge 0}$.  If 
	$\frac{n_1}{a}- \frac{n_2}{b} = 0$, then $$P_{a,b}(X)\sim \frac{A_1A_2}{a^{r_1} b^{r_2}}\frac{r_1 !r_2!}{(r_1+r_2+1)!}\frac{n_1}{a}X^{\frac{n_1}{a}}\ln^{r_1+r_2+1} X.$$
\end{lemma}
\begin{proof}
	We will prove this in three steps.
	
	\textbf{Case 1:  $n_i$ = $1$, $F_1(X) = A_1 X\ln^{r_1} X + o(X\ln^{r_1}X)$, $F_2(X) = A_2 X\ln^{r_2}X+ O(1)$}.\\
	We can assume $a = b =1$. Define $a_n$ to be the number of copies of $n$ in $S_1$, then 
	 $$F_1(X) = \sum_{n\le X} a_n.$$ To simplify, we denote the main term of $F_i(X)$ by $M_i(X)$,  then
	 \begin{equation}
	 \begin{aligned}	
	 P_{1,1}(X) 
	 &=  \sum_{s_1\in S_1} F_2(\frac{X}{s_1}) = \sum_{n\le X} a_n F_2(\frac{X}{n})\\
	 &=  \sum_{n\le X} a_n M_2(\frac{X}{n})+ \sum_{n\le X} a_n O(1).
	 \end{aligned}
	 \end{equation}
	 The last term is easily shown to be small
	 \begin{equation}
	 \begin{aligned}	
	 \sum_{n\le X} a_n O(1) \le O( \sum_{n\le X} a_n) = O(X\ln^{r_1} X).
	 \end{aligned}
	 \end{equation}
	 Assuming $X$ is an integer, we apply summation by parts to compute the first sum 
	  \begin{equation}
	  \begin{aligned}	
	  \sum_{n\le X} a_n M_2(\frac{X}{n}) 
	  = F_1(X) M_2(1) - \int_{1}^{X} F_1(t)\frac{\rd}{\rd t} (M_2(\frac{X}{t}))\rd t.
	  \end{aligned}
	  \end{equation}
	  If $r_2 = 0$, the boundary term is $$A_1A_2X\ln^{r_1} X + o(X\ln^{r_1} X),$$ otherwise it is $0$. The derivative in the integral is
	   \begin{equation}
	   \begin{aligned}	
	   \frac{\rd}{\rd t}(M_2(\frac{X}{t})) & = -A_2 X \frac{1}{t^2} (\ln^{r_2}\frac{X}{t} + r_2 \ln^{r_2-1} \frac{X}{t})\\
	   & = X (\sum_{0\le i\le r_2} P_i(t) \ln^{i}X).
	   \end{aligned}
	   \end{equation}
	   So the integral is
	    \begin{equation}
	    \begin{aligned}	
	   \sum_{0\le i\le r_2} X\ln^{i}X \int_{1}^{X} F_1(t)P_i(t) \rd t.
	    \end{aligned}
	    \end{equation}
	   It is standard in analysis that if $f$ and $g$ are positive and $\lim_{X\to \infty}\int_1^Xf(t)g(t) \rd t= \infty$, then $\int_{1}^{X} o(f(t))g(t) \rd t = o(\int_{1}^{X} f(t)g(t)\rd t )$. Therefore we can plug in $M_1(t)$ for $F_1(t)$ to estimate each integral up to a small error. One can check that for each $i$ the integral of $M_1(t) P_i(t)$ together with $X\ln^i X$ has a main term in the order $X\ln^{r_1+r_2+1} X$. So we can replace $F_1(t)$ by $M_1(t)$ in $(3.3)$ with an error in the order of $o(X\ln X^{r_1+r_2+1})$. Denote the following integral $I$, 
	  \begin{equation}
	  \begin{aligned}	
	  I & = \int_{1}^{X} M_1(t) \frac{\rd}{\rd t}(M_2(\frac{X}{t}))\rd t \\
	  & =  -A_1A_2X\int_{1}^{X}\ln^{r_1} t \cdot (\ln^{r_2}\frac{X}{t} + r_2\ln ^{r_2-1} \frac{X}{t}) \frac{\rd t}{t}.
	  \end{aligned}
	  \end{equation}
	  Using the substitution $u = \frac{\ln t}{\ln X}$, we reduce the integral 
	  \begin{equation}
	  \begin{aligned}	
	  \int_{1}^{X}\ln^{r_1} t \cdot \ln^{r_2}\frac{X}{t}\frac{\rd t}{t} = \ln^{r_1+r_2+1} X\int_{0}^{1} u^{r_1}(1-u)^{r_2} \rd u
	  \end{aligned}
	  \end{equation}
	  to Beta function\cite{Whit} $B(r_1+1,r_2+1)$, therefore
	  \begin{equation}
	  \begin{aligned}	
	  -I =A_1A_2 B(r_1+1,r_2+1)X\ln^{r_1+r_2+1} X + o (X(\ln X)^{r_1+r_2+1}).	  
	  \end{aligned}
	  \end{equation}
	  This is always of greater order than the boundary term, and hence finishes the proof of the first case. 
	  
	  \textbf{Case 2: $n_i = 1$,  $F_i(X) = A_i X \ln^{r_i} X + o(X\ln^{r_i}X)$}.\\
	  For any $\epsilon$, we can bound $F_i(X)$ by $A_iX\ln^{r_i}X(1+\epsilon) + O_{\epsilon}(1). $ Therefore we can bound 
	  $$\limsup_{X\to \infty} \frac{P_{1,1}(X)}{X\ln^{r_1+r_2+1}X} \le (1+\epsilon)^2A_1A_2 B(r_1+1,r_2+1),$$ by Case 1. Similarly we can bound 
	  $$\liminf_{X\to \infty} \frac{P_{1,1}(X)}{X\ln^{r_1+r_2+1}X} \ge (1-\epsilon)^2A_1A_2B(r_1+1,r_2+1).$$
	  So the limit exists and has to be $A_1A_2B(r_1+1,r_2+1)$. In case where some $A_i = 0$, we only need the upper bound to show the limit is $0$. 
	  	  
	  \textbf{ General case:} \\
	  Generally, we consider all possible $a$ and $b$. The condition $s_1^a s_2^b\le X$ is equivalent to $s_1^{n_1} s_2^{n_2} \le X^{n_1/a} = X^{n_2/b}$. The distribution of $s_i^{n_i}$ is
	   \begin{equation}
	   \begin{aligned}	
	   F_i(X^{1/n_i}) = \frac{A_i}{n_i^{r_i}} X\ln^{r_i} X + o(X\ln^{r_i}X),
	   \end{aligned}
	   \end{equation}
	   and we can regard $\frac{A_i}{n_i^{r_i}}$ as the new coefficients. The general distribution is the product distribution in Case 2 when one plugs in $X^{n_1/a}$,
	  \begin{equation}
	  \begin{aligned}	
	  P_{a,b}(X) & = \frac{A_1}{n_1^{r_1}}\frac{A_2}{n_2^{r_2}} B(r_1+1, r_2+1) (\frac{n_1}{a})^{r_1+r_2+1}X^{n_1/a} (\ln X)^{r_1+r_2+1} + o(X^{n_1/a}(\ln X)^{r_1+r_2+1})\\
	  & \sim \frac{A_1}{a^{r_1}} \frac{A_2}{b^{r_2}} B(r_1+1, r_2+1) \frac{n_1}{a} X^{n_1/a} (\ln X)^{r_1+r_2+1}.
	  \end{aligned}
	  \end{equation}
	  \end{proof}
	  
	  \begin{lemma}\label{proneq}
	  	Let $F_i(X)$, $i=1,2$ be as given above, $F_i(X) \sim A_i X^{n_i}\ln^{r_i} X$ where $0 <n_i\le 1$ and $r_i\in \mathbb{Z}_{\ge 0}$.  If 
	  	$\frac{n_1}{a}- \frac{n_2}{b} > 0$, then there exists a constant $C$ such that  $$P_{a,b}(X)\sim C X^{\frac{n_1}{a}}\ln^{r_1} X.$$
	  	Furthermore if $F_i(X) \le A_i X^{n_i}\ln^{r_i} X$, then we have
	  	$$P_{a,b}(X)\le A_1A_2 \frac{r_2!}{b^{r_2}a^{r_1}} \frac{1}{(\frac{n_1}{a}-\frac{n_2}{b})^{r_2+1}}\frac{n_1}{a}X^{\frac{n_1}{a}} \ln^{r_1} X.$$
	  \end{lemma}
	  \begin{proof}
	  	We first prove the existence of $C$ in two steps. 
	  	
	  	 \textbf{Case 1:  $F_1(X) = A_1 X^{n_1}\ln^{r_1} X + O(1)$, $F_2(X) = A_2 X^{n_2}\ln^{r_2}X+ o(X^{n_2}\ln^{r_2}X)$}.\\
	  	As in Lemma \ref{proeq}, we need to bound the sum
	  	 \begin{equation}
	  	 \begin{aligned}	
	  	 P_{a,b}(X) 
	  	 &=  \sum_{n^am^b\le X} a_n b_m = \sum_{m^b\le X} b_m F_1(\frac{X^{1/a}}{m^{b/a}})\\
	  	 &= \sum_{m^b\le X} b_m A_1(\frac{X^{1/a}}{m^{b/a}})^{n_1} \ln^{r_1}(\frac{X^{1/a}}{m^{b/a}}) + \sum_{m^b\le X} b_m O(1)\\
	  	 &=\frac{A_1}{a^{r_1}} X^{n_1/a}\ln^{r_1} X \sum_{m^b\le X} \frac{b_m}{m^{bn_1/a}}(1-\frac{\ln m^b}{\ln X})^{r_1} + O(X^{n_2/b}\ln^{r_2} X).\\
	  	 \end{aligned}
	  	 \end{equation}
	  	 It suffices to show the sum $$C(X) = \sum_{m^b\le X} \frac{b_m}{m^{bn_1/a}}(1-\frac{\ln m^b}{\ln X})^{r_1},$$ converges to a constant $C'$, i.e., $C(X) = C' + o(1)$. Notice that $C(X)$ is monotonically increasing, so it suffices to show $C(X)$ is bounded. We will assume $X$ to be integral for simplicity, by summation by parts,
	  	 \begin{equation}
	  	 \begin{aligned}	
	  	 C(X) & \le \sum_{m^b\le X} \frac{b_m}{m^{bn_1/a}} = \frac{F_2(X^{1/b})}{X^{n_1/a}} + \frac{bn_1}{a}\int_1^{X^{1/b}} F_2(t) t^{-bn_1/a-1} \rd t\\
	  	 & \le O(X^{n_2/b-n_1/a}) + \frac{bn_1}{a}\int_1^{X^{1/b}} (M t^{n_2}\ln^{r_2} t + M)t^{-bn_1/a-1} \rd t,
	  	 \end{aligned}
	  	 \end{equation}
	  	 is bounded by a constant. The first term is $o(1)$ since $\frac{n_1}{a}- \frac{n_2}{b} > 0$. For the second term, we can always find $M$ such that $F_2(t) \le Mt^{n_2}\ln^{r_2}t + M$, where the constant term $M$ is a technical modification when $t=1$. One can compute the integral to see that it is bounded by a constant. Therefore, we have proved that $C(X) = C' + o(1)$ and $$P_{a,b}(X) \sim \frac{A_1C'}{a^{r_1}} X^{n_1/a} \ln^{r_1} X.$$
	  	 
	  	 \textbf{Case 2:  $F_i(X) = A_i X^{n_i}\ln^{r_i} X + o(X^{n_i}\ln^{r_i}X)$}.\\
	  	 Notice that $C(X)$ is purely dependent on $F_2(X)$ and independent of $F_1(X)$ once we have decided on these constants $r_i$, $n_i$ and $a$, $b$. Therefore the coefficient of the main term of $P_{a,b}$ is linearly dependent on $A_1$.
	  	 
	  	 To get the upper bound, we can bound $F_1(X)\le A_1(1+\epsilon)X^{n_1}\ln^{r_1} X + O_{\epsilon}(1)$ by definition and compute the upper bound of $P_{a,b}(X)$,
	  	 $$\limsup_{X\to \infty} \frac{P_{a,b}(X)}{X^{n_1/a}\ln^{r_1}X} \le (1+\epsilon)\frac{A_1}{a^{r_1}}C'$$
	  	 by Case 1. Similarly, we can deal with the lower bound. Therefore, 
	  	 $$\lim_{X\to \infty} \frac{P_{a,b}(X)}{X^{n_1/a}\ln^{r_1}X} = \frac{A_1}{a^{r_1}}C'$$
	  	 which proves the general case with $C=\frac{A_1C'}{a^{r_1}}$.\\
	  	 
	  	 \textbf{Bound on $C$:}\\
	  	 Next we assume further that $F_i(X)$ are bounded by $M_i(X) = A_i X^{n_i}\ln^{r_i} X$. We want to show the constant $C$ can be bounded by $O(A_1A_2)$. By summation by parts, 
	  	  \begin{equation}
	  	  \begin{aligned}	
	  	  P_{a,b}(X) 
	  	  & \le \sum_{n\le X^{1/a}} a_n M_2(\frac{X^{1/b}}{n^{a/b}})\\
	  	  & \le F_1(\lfloor X^{1/a}\rfloor) M_2(1) - \int_{1}^{\lfloor X^{1/a}\rfloor} M_1(t) \frac{\rd}{\rd t}(M_2(\frac{X^{1/b}}{t^{a/b}}))\rd t.
	  	  \end{aligned}
	  	  \end{equation}
	  	   	If $r_2 = 0$, the boundary term is bounded by $$\frac{A_1A_2}{a^{r_1}}X^{n_1/a} \ln^{r_1} X,$$ otherwise it is $0$. Consider the following integral
	  	   	\begin{equation}
	  	   	\begin{aligned}	
	  	   	-I & = -\int_{1}^{\lfloor X^{1/a}\rfloor} M_1(t) \frac{\rd}{\rd t}(M_2(\frac{X}{t}))\rd t \\
	  	   	& =  A_1A_2X^{\frac{n_2}{b}} (\frac{a}{b})\int_{1}^{\lfloor X^{1/a}\rfloor}t^{n_1 - \frac{a}{b}n_2 }\ln^{r_1} t \cdot (\frac{n_2}{b^{r_2}}\ln^{r_2}\frac{X}{t^a} + \frac{r_2}{b^{r_2-1}}\ln ^{r_2-1} \frac{X}{t^a}) \frac{\rd t}{t}\\
	  	   	& \le A_1A_2X^{\frac{n_2}{b}} (\frac{1}{a^{r_1}b^{r_2}})\int_{1}^{X}t^{\frac{n_1}{a} - \frac{n_2}{b} }\ln^{r_1} t \cdot (\frac{n_2}{b}\ln^{r_2}\frac{X}{t} + r_2\ln ^{r_2-1} \frac{X}{t}) \frac{\rd t}{t}.
	  	   	\end{aligned}
	  	   	\end{equation}
	  	   	The integral is a sum of multiple pieces in the form of
	  	   	$$I_{n,r_1,r_2} = \int_{1}^{X} t^{n}\ln^{r_1}t \ln^{r_2} \frac{X}{t} \frac{\rd t}{t}.$$
	  	   	It satisfies an induction formula
	  	   	\begin{equation}
	  	   	\begin{aligned}	
	  	   	I_{n,r_1,r_2} = -\frac{r_1}{n} I_{n, r_1-1,r_2} + \frac{r_2}{n} I_{n,r_1, r_2 -1}    \end{aligned}
	  	   	\end{equation}
	  	   	with initial data
	  	   	$$I_{n, r_1, 0} \le \frac{1}{n} X^{n} \ln^{r_1} X$$
	  	   	$$I_{n, 0, r_2} \le \frac{r_2!}{n^{r_2+1}} X^{n}.$$
	  	   	Notice that $I_{n,r_1,r_2}$ is always positive, by the induction formula one can show
	  	   	\begin{equation}
	  	   	\begin{aligned}	
	  	   	I_{n,r_1,r_2} \le \frac{r_2!}{n^{r_2+1}} X^n \ln^{r_1}X.
	  	   	\end{aligned}
	  	   	\end{equation}
	  	   	If $r_2 = 0$, $-I$ together with the boundary term is bounded,
	  	   	\begin{equation}
	  	   	\begin{aligned}	
	  	   	P_{a,b}(X) \le \frac{A_1A_2}{a^{r_1}} \frac{n_1}{a} \frac{1}{\frac{n_1}{a}-\frac{n_2}{b}} X^{\frac{n_1}{a}}\ln^{r_1} X.
	  	   	\end{aligned}
	  	   	\end{equation}
	  	   	When $r_i \ne 0$, we have 
	  	   	\begin{equation}
	  	   	\begin{aligned}	
	  	   	P_{a,b}(X) \le A_1A_2 \frac{r_2!}{b^{r_2}a^{r_1}} \frac{n_1}{a}\frac{1}{(\frac{n_1}{a}-\frac{n_2}{b})^{r_2+1}}X^{\frac{n_1}{a}} \ln^{r_1} X. 
	  	   	\end{aligned}
	  	   	\end{equation}
	  	   	This formula is compatible with the special cases where $r_i$ could be $0$. 
	  \end{proof}

	  \begin{corollary}
	  	Let $k$ be an arbitrary number field, and $G_1\subset S_n$ and $G_2\subset S_m$ be two Galois groups with nontrivial isomorphic quotient. Suppose Malle's conjecture holds for both groups, then there is a lower bound on $N(G_1\times G_2\subset S_{mn},X)$ that
	  $$N(G_1\times G_2\subset S_{mn},X) \ge CX^{a}\ln^{r} X + o(X^a\ln^r X),$$
	  where $a = max\{ a(G_1)/m, a(G_2)/n\}$.  If $ a(G_1)/m= a(G_2)/n$, then $r = b(G_1, k)+ b(G_2,k)-1$; if $ a(G_1)/m > a(G_2)/n$, then $r = b(G_1,k) -1$.
	  \end{corollary}
	   A lower bound $X^a$ is also obtained in \cite{Mal02} Proposition $4.2$. Here we improve on the lower bound by adding a $\ln^r X$ factor. 
       \section{Uniformity Estimate for $S_n$ and $A$ number fields}
       In this section, we are going to include and prove some necessary uniformity results we need for $S_3$ cubic, $S_4$ quartic, $S_5$ quintic and $A$ number fields over arbitrary global field $k$. 
       \subsection{Local uniformity for $S_n$ extensions for $n = 3, 4$}
       We will include the uniformity estimates for $S_3$ and $S_4$ extensions with certain ramification behavior at finitely many places. Both results are deduced by class field theory. 
       
       For totally ramified $S_3$ cubic extensions, we have Proposition 6.2 from \cite{DW88}:
       \begin{theorem}\label{uni3}
       	The number of non-cyclic cubic extensions over $k$ which are totally ramified at a product of finite places $q = \prod{p_i}$ is:
       	$$N_q(S_3,X) = O(\frac{X}{|q|^{2-\epsilon}}),$$
       	for any number field $k$ and any square free integral ideal $q$. The constant is independent of $q$, and only depends on $k$. 
       	\end{theorem}      	
       	
       	For discussions about overramified $S_4$ quartic extensions, we will follow the definition of \cite{Bha05}: $p$ is overramified if $p$ factors into $P^4$, $P^2$ or $P_1^2P_2^2$ for a finite place $p$ and if $p$ factors into a product of two ramified places for infinite place. Equivalently, this means the inertia group at $p$ contains $\large \langle(12)(34) \large \rangle$ or $\large \langle(1234) \large \rangle$. The uniformity estimate for overramified $S_4$ extensions over $\mathbb{Q}$ is given in \cite{Bha05}, see Proposition 23. And we are going to prove the same uniformity over an arbitrary number field $k$ by the same method. Let $K_{24}$ be an $S_4$ extension over $k$. Denote $K_6$ and $K_3$ to be the subfields corresponding to the subgroup $E = \{(e, (12), (34), (12)(34))\}$ and $H=\large \langle E, (1234)\large \rangle$.     	
       	\begin{theorem}\label{uni4}
       	The number of $S_4$ quartic extensions over $k$ which are overramified at a product of finite places $q = \prod{p_i}$ is:
       		$$N_q(S_4,X) = O(\frac{X}{|q|^{2-\epsilon}}),$$
       	for any number field $k$ and any square free integral ideal $q$. The constant is independent of $q$, and only depends on $k$.        		
       	\end{theorem}
       	\begin{proof}
       	We can apply the class field theory argument in \cite{Bha05}. On one hand, over arbitrary $k$ we still have that $\Nm_{K_3/k}(\disc(K_6/K_3))$ is a square ideal in $k$ for any $S_4$ extension. Actually 
       	$$\Nm_{K_3/k}(\disc(K_6/K_3)) =\Disc(K_6)/\Disc(K_3)^2,$$ which is the Artin conductor associated to the character $\chi = \text{Ind}^G_E -2\cdot \text{Ind}^G_H$ where $E$ and $H$ are corresponding subgroups of $K_6$ and $K_3$. Here $\text{Ind}^G_E$ is the induced character of the identity character of $E$ as a subgroup of $G = S_4$. By computation, the character $\chi$ has value $-4$ at the conjugacy class of $(12)(34)$, and $-2$ at $(1234)$. The character values are even and so the Artin conductor is always a square. On the other hand, we still have the result on the mean $2$-class number of non-cyclic cubic extensions over any number field $k$ in \cite{BSW15}. It follows that the summation of $2$-class number is $O(X)$ over non-cyclic cubic extensions with bounded discriminant. 
       	\end{proof}
       	\subsection{Local uniformity for $S_n$ extensions for $n=5$}
       	In this section, we are going to prove the uniformity of $S_5$ extensions by geometry of numbers based on previous works \cite{Bha10, Bhasieve, BSW15}. We will use slightly different notation just for this section. Denote $K$ to be an arbitrary number field with degree $d = \deg (K)$. For a certain scheme $Y \in \mathbb{A}^n_\mathbb{Z}$, let $k$ be its codimension.  
       	\begin{theorem}\label{uni5}
       		The number of $S_5$ quintic extensions over $K$ which are totally ramified at a product of finite places $q = \prod{p_i}$ is:
       		$$N_q(S_5,X) = O(\frac{X}{|q|^{4/15-\epsilon}}),$$
       		for any number field $K$ and any square free integral ideal $q$. The constant is independent of $q$, and only depends on $k$. 
       	\end{theorem}  
       	    	
       The proof is an application of Bhargava's geometric sieve method \cite{Bhasieve}. By \cite{Bhasieve}, the points in the prehomogenous space with certain ramification at a finite place $p$ are $O_K/pO_K$-points on a certain scheme $Y$, which is cut out by partial derivatives of the discriminant polynomial. And to get a power saving error, we can apply the averaging technique like in \cite{BBP, BST, S5power} as suggested in Remark $4.2$ in \cite{Bhasieve}. Instead of considering points that have extra ramification at primes greater than $M$, we only need to look at the number of points that have extra ramification at specified primes $q = \prod p$. So we will first determine the number of $O_K/qO_K$-points of a scheme $Y$ in an expanding ball and then compute the number of lattice points in the fundamental domain by averaging technique. We first look at the case when $K$ is $\mathbb{Q}$.  Corresponding to Theorem $3.3$ in \cite{Bhasieve}, we have the following theorem. 
       
       	\begin{theorem}\label{rBoverq}
       		Let $B$ be a compact region in $\mathbb{R}^n$ having finite measure. Let $Y$ be any closed subscheme of $\mathbb{A}^n_{\mathbb{Z}}$ of codimension $k$. Let $r$ be a positive real number and $q$ be a square free integer. Then we have 
       		$$\sharp \{  a\in rB\cap \mathbb{Z}^n \mid a(\text{mod }q) \in Y(\mathbb{Z}/q\mathbb{Z})\} =  O( r^{n-k} ) \cdot C^{\omega(q)} \cdot \max \{  1, (\frac{r}{q})^k \},$$
       		where the implied constant depends only on $B$ and $Y$, and $C$ is an absolute constant only depending on $Y$. 
       	\end{theorem} 
       	\begin{proof}
       			 The case when $k =0$ is trivial since the number of lattice points in the box is $O(r^n)$. So the initial case is $k = 1$ with $n=1$ . Then there is only one polynomial $f(x)$ for $n=1$. The number of points is $O(C^{\omega(q)}\cdot \max \{ 1, \frac{r}{q}\})$ where we could choose $C$ to be the degree of $f$ and the implied constant depends on $f$ and $B$. 
       			 
       			 We will apply induction on $n$ and $k$. Let $\pi : \mathbb{A}^n_{\mathbb{Z}} \to \mathbb{A}^{n-1}_{\mathbb{Z}}$ be the projection onto the first $n-1$ coordinates. By dimension formula, the image $\bar{Y}$ of $Y$ in $\mathbb{A}^{n-1}_{\mathbb{Z}}$ is a closed subscheme with codimension at least $k-1$. And we can choose $\pi$ carefully so that for each $y = (a_1, \cdots, a_{n-1})\in \mathbb{Z}^{n-1}$ that $y (\text{mod } q) \in \bar{Y}(\mathbb{Z}/q\mathbb{Z}) $, the number of lattice points lying in the fiber is
       			 $$\sharp \{ a =  (a_1, \dots, a_{n-1}, b) \in rB\cap \mathbb{Z}^n \mid  a (\text{mod } q) \in Y(\mathbb{Z}/q\mathbb{Z}) \},$$  			 
       			 and is bounded by $C^{\omega(q)}\cdot \max \{ 1, \frac{r}{q}\}$. Indeed suppose $f\in \mathbb{Z}[x_1, \cdots, x_n]$ vanishes on $Y$, and $s$ is the direction of projection, then $f(v+st)$ as a polynomial in $t$ has leading coefficients as a polynomial in $s$. So if we choose $s$ such that the leading coefficients is non-zero, then aside from finitely many $p$, the number of solutions in  $\mathbb{Z}/p\mathbb{Z}$ at a fixed $v$ is bounded by the degree of $f$. Therefore, the number of solutions in $\mathbb{Z}/q\mathbb{Z}$ is at most $O(C^{\omega(q)})$ where $C$ is the degree of $f$ and the implied constant depends on the bad primes. And the number of lattice points follows by the induction to $n=1$ case. 
%
       		
       			 By induction, the number of $y\in \mathbb{Z}^{n-1}$ in the projection of $rB$ and in $\bar{Y}(\mathbb{Z}/q\mathbb{Z})$ is $O(r^{n-k})\cdot C^{\omega(q)}\cdot \max \{1, (\frac{r}{q})^{k-1} \}$, and the number of $x_n$ for each $y$ is $C^{\omega(q)}\cdot \max \{1,\frac{r}{q}     \}$. So the totaly estimate is      			 
       			 \begin{equation}
       			 \begin{aligned}	
       			 &\sharp \{ a\in rB\cap \mathbb{Z}^n \mid  a (\text{mod } q) \in Y(\mathbb{Z}/q\mathbb{Z}) \}  \\
       			 = & O(r^{n-k} ) \cdot C^{\omega(q)} \cdot \max \{  1, (\frac{r}{q})^{k-1} ,  \frac{r}{q}, (\frac{r}{q})^k\} = O(r^{n-k} ) \cdot C^{\omega(q)} \cdot \max \{ 1, (\frac{r}{q})^k\}.
       			 \end{aligned}
       			 \end{equation}	
       	\end{proof}
       	Notice that although Theorem $3.3$ in \cite{Bhasieve} deals with all $p> M$, it can also give an upper bound for counting at a single prime. On one hand, our statement includes the cases where finitely many ramification conditions are specified. On the other hand, as suggested by Bhargava, we can get a slightly better error of order $r^{n-k}$ instead of $r^{n-k+1}$. 
       	
       	
       	In order to apply the averaging technique, we also need to consider the number of lattice points in the box $mrB$ that is not necessarily expanding homogeneously in each direction. Here $m$ is a lower triangle unipotent transformation in $GL_n(\mathbb{Q})$ which does not change the estimate much. And $r = (r_1, \dots, r_n)$ is the scaling factors and the estimate will depend on $r_i$. 
       	\begin{theorem}\label{mrBoverq}
       		Let $B$ be a compact region in $\mathbb{R}^n$ having finite measure. Let $Y$ be any closed subscheme of $\mathbb{A}^n_{\mathbb{Z}}$ of codimension $k$. Let $r = (r_1, \dots, r_n)$ be a diagonal matrix of positive real number where $r_i \ge \kappa$ for a certain $\kappa$, $q$ be a square free integer, and $m$ be a lower triangle unipotent transformation in $GL_n(\mathbb{R})$. Then we have 
       		$$\sharp \{  a\in mrB\cap \mathbb{Z}^n \mid a(\text{mod }q) \in Y(\mathbb{Z}/q\mathbb{Z})\} = O(\frac{\prod_{i=1}^n r_i}{q^k}) \cdot C^{\omega(q)} \cdot \max\{1, \frac{q}{r_i}, \frac{q^2}{r_i r_j} , \dots, \frac{q^k}{ \prod_{i= i_1}^{i_k} r_i}\},$$
       		where the implied constant depends only on $B$, $Y$ and $\kappa$, and $C$ is an absolute constant only depending on $Y$. 
       	\end{theorem} 
       	\begin{proof}
       		For case $k =0$, we can get the result $O(\prod_{i=1}^n r_i)$ directly because the total count of lattice points in $mrB$ only differs with those in $rB$ by lower dimension projections of $rB$ which could be bounded by $O(\prod_{i=1}^n r_i)$ where the implied constant depends on $\kappa$. 
       		
       		
       		The initial case when $k=1$, $ n= 1$ is estimated to be $O(\frac{r_1}{q}) \cdot C^{\omega(q)} \cdot \max\{1, \frac{q}{r_1}\}$. It is the same with Theorem \ref{rBoverq} since there is no non-trivial unipotent action. For general $n$ and $k$, we will still consider the projection to the first $n-1$ coordinates. By induction, the number of points in $\bar{Y}$  is at most $O(\frac{\prod_{i=1}^{n-1} r_i}{q^{k-u}}) \cdot C^{\omega(q)} \cdot \max\{1, \frac{q}{r_i}, \frac{q^2}{r_i r_j} , \dots, \frac{q^{k-u}}{ \prod_{i= i_1}^{i_{k-u}} r_i}\}$. And for a fixed $y=(a_1, \dots, a_{n-1})\in \mathbb{Z}^{n-1}$, the number of lattice points lying in the fiber is
       		\begin{equation}
       		\begin{aligned}	
       		&\sharp \{ a =  (a_1, \dots, a_{n-1}, b) \in mrB\cap \mathbb{Z}^n \mid  a (\text{mod } q) \in Y(\mathbb{Z}/q\mathbb{Z}) \}  \\
       		=& \sharp \{ b \in P_{y}(mrB)\cap \mathbb{Z} \mid  a (\text{mod } q) \in Y(\mathbb{Z}/q\mathbb{Z}) \}.\\
       		\end{aligned}
       		\end{equation}	
       		Here $P_{y}(R)$ means the section of $R$ with $y = (a_1, \dots, a_{n-1})$ fixed where $R$ is any compact region. A lower triangle unipotent transformation $m$ has the property that once $x_i$ is fixed for $i<k$, then the action on $x_k$ is just a translation. Therefore there exists $y'$ such that $P_{y}(mR)$ and $P_{y'}(R)$ only differ by a constant translation, i.e.,  $P_{y}(mR) = P_{y'}(R) + b_0$ where $b_0$ is a constant vector.  Since the estimate only depends on the compact region in terms of its low dimension projection, constant translation will not affect the estimate, so we can look at instead
       		\begin{equation}
       		\begin{aligned}	
       		& \sharp \{ b  \in P_{y'}(rB)\cap \mathbb{Z}^k \mid a (\text{mod } q) \in Y(\mathbb{Z}/q\mathbb{Z}) \}  \\
       		= & O(\frac{r_n}{q}) \cdot C^{\omega(q)} \cdot \max\{1, \frac{q}{r_n} \}.\\
       		\end{aligned}
       		\end{equation} 
       		The implied constant in the last equality could be bound uniform for all $y$ by similar argument in Theorem \ref{rBoverq}. Therefore by taking the product, we get
       		\begin{equation}
       		\begin{aligned}	
       		&\sharp \{ a\in mrB\cap \mathbb{Z}^n \mid  a (\text{mod } q) \in Y(\mathbb{Z}/q\mathbb{Z}) \}  \\
       		= & O(\frac{\prod_{i=1}^n r_i}{q^k}) \cdot C^{\omega(q)} \cdot \max\{1, \frac{q}{r_i}, \frac{q^2}{r_i r_j} , \dots, \frac{q^k}{ \prod_{i= i_1}^{i_k} r_i}\},\\
       		\end{aligned}
       		\end{equation}	
       		and the implied constant depends only on $B$, $Y$ and $\kappa$.      		
       	\end{proof}
       	\begin{remark}
         We can consider the above theorem as an improvement on Theorem $26$ \cite{BST} in this special case. Indeed, the cubic rings $K$ that are ramified at $p$ with $p^k | \Disc (K)$ are a union of $O(p^{4-k})$ translation of lattices. So we basically prove that when we count these lattice points in the expanding ball $mrB$, we do not get those error terms at the tail in line $(29)$ in \cite{BST}.
         \end{remark}
         
          \begin{proof}[\textbf{Proof of Theorem \ref{uni5} over $\mathbb{Q}$}]
          	
          We first prove this statement over $\mathbb{Q}$ and then will show that the computation over other number field $K$ should give the same answer. Recall that the quintic order is parametrized by $G(\mathbb{Z})$-orbits in $V(\mathbb{Z})$ where $G = GL_4\times GL_5$ and $V$ is the space of quadruples of skew symmetric $5\times 5$ matrices. Denote the fundamental domain of $G(\mathbb{R})/ G(\mathbb{Z})$ by $\mathcal{F}$ and $B$ is a compact region in $V(\mathbb{R})$. Let $S$ be any $G(\mathbb{Z})$-invariant subset of $V^{(i)}_{\mathbb{Z}}$ which specifies a certain property of quintic orders, $S^{irr}$ be the subset of irreducible points in $S$, and $N(S; X)$ denotes the number of irreducible-$G(\mathbb{Z})$ orbits in $S$ with discriminant less than $X$. Then by formula $(20)$ in\cite{BST}, the averaging integral for a certain signature $i$ is
          		\begin{equation}\label{bstave}
          		\begin{aligned}
          		N(S; X) = \frac{1}{M_i}\int_{g\in \mathcal{F}} \sharp \{x \in S^{irr} \cap gB\cap V_{\mathbb{R}}^{(i)}: |\text{Disc}(x)|<X \}  dg
          		\end{aligned}
          		\end{equation}          		
          		where $M_i$ is a constant depending on $B$. 
          		
          		Here for our purpose, $S = S_q$ should be the set of maximal orders that are totally ramified at all primes $p| q$. In order to apply Theorem \ref{mrBoverq}, we can replace the condition $x\in S^{irr}$ by $x\in Y(\mathbb{Z}/q\mathbb{Z})$ where $Y$ is a codimension $k = 4$ variety in a $40$ dimensional space defined by $f^{(j)} =0$ for all partial derivatives of the discriminant polynomial with order $j<4$. See \cite{Bhasieve} for the definition of $Y$. 
          		
          		For $g\in G(\mathbb{R})$, we have $g= mak\lambda$ as the Iwasawa decomposition \cite{Bha10}. Here $m$ is an lower triangle unipotent tranformation, $a = (t_1, \dots, t_n)$ is a diagonal element with determinant 1 and $k$ is an orthogonal transformation in $G(\mathbb{R})$ and $\lambda = \lambda I$ is the scaling factor. We will choose $B$ such that $KB = B$, so $gB = ma\lambda B = m r B$,  in which $r = \lambda (t_1, \dots, t_n)$ satisfies that $\prod^n_1 t_i =1$. Lastly, the requirement $|\text{Disc}(x)|<X $ could be dropped as long as we take $\lambda \le O(X^{1/d})$ where this implied constant depends only on $B$. So we have        		
          		 $$\sharp \{x \in S^{irr} \cap gB\cap V_{\mathbb{R}}^{(i)}: |\text{Disc}(x)|<X \}\le \sharp \{  x \in mrB\cap \mathbb{Z}^n \mid a(\text{mod }q) \in Y(\mathbb{Z}/q\mathbb{Z})\}. $$
          		
          		We are going to apply Theorem \ref{mrBoverq} to estimate the integral in (\ref*{bstave}). By \cite{Bha10}, all $S_5$ orders are parametrized by quadruples of skew symmetric $5\times 5$ matrices. So there are $40$ variables and therefore the dimension for the whole space is $n=40$. Let's call those variables $a^l_{ij}$ where $1\le l \le 4$ means the $m$-th matrix, $1\le i \le 4$ is the row index of a skew-symmetric $5\times 5$ matrix, $2\le j\le 5$ is the column index. We can define the partial order among all $40$ entries: $a^i_{jk}$ is smaller than $a^l_{mn}$ if $i\le l$, $j\le m$ and $k\le n$. The scaling factor $t_i$ in our situation could be described by a pair of diagonal matrices $(A , B)$ where $$A = \text{diag} (s_1^{-3}s_2^{-1}s_3^{-1}, s_1s_2^{-1}s_3^{-1}, s_1s_2s_3^{-1}, s_1s_2s_3^{3} )$$ and $$B =\text{diag} (s_4^{-4}s_5^{-3}s_6^{-2}s_7^{-1},  s_4s_5^{-3}s_6^{-2}s_7^{-1},  s_4s_5^{2}s_6^{-2}s_7^{-1} ,  s_4s_5^{2}s_6^{3}s_7^{-1},  s_4s_5^{2}s_6^{3}s_7^{4} ). $$ Then $t_{lij} = A_l B_i B_j$ is the scaling factor for the $a^l_{ij}$ entry. Since the fundamental domain requires that all $s_i \ge C$, this partial order also gives the partial order on the magnitude of $r_{lij}= \lambda t_{lij}$. 
          		
          		There are many regions in the fundamental domain that provides irreducible $S_5$-orders. We will consider the biggest region first, i.e., the points with $a^1_{12} \ne 0$. This region requires that $\lambda s_1^{-3}s_2^{-1}s_3^{-1}s_4^{-3}s_5^{-6}s_6^{-4}s_7^{-2} \ge \kappa$, therefore $r_{lij} \ge \kappa$ for all $l, i,j$. Let us denote this region in $\mathcal{F}$ to be  $D_\lambda = \{ s_i\ge C_i\mid s_1^{3}s_2s_3s_4^{3}s_5^{6}s_6^{4}s_7^{2}\le \lambda /\kappa\}$. So we could apply Theorem \ref{mrBoverq} directly. Let's call this count $N^1( Y ; X)$. The corresponding integrand, i.e., the number of lattice points in the expanding ball $gB$ where $g\in D_{\lambda}$ is bounded by
          		\begin{equation}
          		\begin{aligned}           		
          		L^1 =  &\sharp \{x \in mrB\cap V^{(i)}_{\mathbb{Z}} \mid  x (\text{mod } q) \in Y(\mathbb{Z}/ q\mathbb{Z})\}\\
          		= &O(\frac{\lambda^n}{q^k}) \cdot C^{\omega(q)} \cdot \max\{1, \frac{q}{\lambda t_i}, \frac{q^2}{\lambda ^2t_i t_j} , \dots, \frac{q^k}{\lambda ^k \prod_{i= i_1}^{i_k} t_i}\}\\    		 
          		= &  O(\frac{\lambda^{40}}{q^4}) \cdot C^{\omega(q)} \cdot \max\{1, \frac{q}{\lambda t_{112}}, \frac{q^2}{\lambda^2 t_{112} t_{113}}, \frac{q^2}{\lambda^2 t_{112} t_{212}},  \frac{q^3}{\lambda^3 t_{112} t_{113} t_{123}} , \frac{q^3}{\lambda^3 t_{112} t_{113} t_{114}}, \\
          		& \frac{q^3}{\lambda^3 t_{112} t_{113} t_{212}}, \frac{q^3}{\lambda^3 t_{112} t_{212} t_{312}},  \frac{q^4}{\lambda^4 t_{112} t_{113} t_{114}t_{123}}, \frac{q^4}{\lambda^4 t_{112} t_{113} t_{114}t_{212}}, \frac{q^4}{\lambda^4 t_{112} t_{113} t_{123}t_{212}},\\& \frac{q^4}{\lambda^4 t_{112} t_{113}t_{212}t_{213}}, \frac{q^4}{\lambda^4 t_{112} t_{113}t_{212}t_{312}} , \frac{q^4}{\lambda^4 t_{112} t_{212}t_{312}t_{412}} \}. \\
          		\end{aligned}
          		\end{equation}
          		To integrate $L^1$ over $D_\lambda$ and then against $\lambda$, we just need to focus on the inner integral over $D_\lambda$, and see whether the integral of those product of $t_{lij}$ over $D_\lambda$ produces $O(1)$ or $\lambda^r$ for some $r\ge 0$ as the result.  If it is $O(1)$, then we just need to integrate against $\lambda$ and get the expected estimate, i.e., $\frac{X^{40-i}}{q^i}$ for $0\le i\le 4$ where $i$ is the number of $t_{lij}$ factors in the product; if it is $\lambda^r$ for some power $r>0$, then we will get a bigger power of $X$. 
          		
          		For example, $t^{-1}_{112} = s_1^{3}s_2s_3s_4^{3}s_5^{6}s_6^{4}s_7^{2}$ and $dg = \delta_5 ds^{\times} = s_1^{-8}s_2^{-12}s_3^{-8}s_4^{-20}s_5^{-30}s_6^{-30}s_7^{-20} ds^{\times}$, therefore $t^{-1}_{112} \delta _5$ contains $s_i$ with negative power for each $i$. So after integrating over $D_\lambda$, it is $O(1)$. Same thing holds for all other products listed as above except: $t_{112} t_{113} t_{123}$, $t_{112} t_{113}t_{114}$, $t_{112} t_{113}t_{114}t_{123}$, $t_{112} t_{113}t_{114}t_{212}$, $t_{112} t_{113}t_{123}t_{212}$. All these products have at most $4$ $t_{lij}$ factors, so the biggest power we could get for $s_4$, $s_5$, $s_6$ and $s_7$ should be $(B_1B_2)^4 = s_4^{-12}s_5^{-24}s_6^{-16}s_7^{-8}$, so those later $s_i$ is never a problem. 
          		
          		Among the product with $3$ factors, the $s_i$ part for small $i$ in $t_{112} t_{113} t_{123}$ and $t_{112} t_{113}t_{114}$ is $s_1^{-9}s_2^{-3}s_3^{-3}$. Since $s_1\le O(\lambda^{1/3})$, the integral over $D_{\lambda}$ should be $O(\lambda^{1/3})$. Among the product with $4$ factors, $t_{112} t_{113}t_{114}t_{212}$ and $t_{112} t_{113}t_{123}t_{212}$ has factor $s_1^{-8}s_2^{-4}s_3^{-4}$, while $t_{112} t_{113}t_{114}t_{123}$ has a bigger term $s_1^{-12}s_2^{-4}s_3^{-4}$, whose integral ends up being $O(\lambda^{4/3})$. 
          		
          		So the whole result is:
          		\begin{equation}
          		\begin{aligned}
          		N^{1}(Y; X) &  \le \frac{1}{M_i} \int^{O(X^{1/40})}_{\lambda = O(1)} \int_{D_{\lambda}} L^1 s_1^{-8}s_2^{-12}s_3^{-8}s_4^{-20}s_5^{-30}s_6^{-30}s_7^{-20}\text{d}s^{\times} \text{d}\lambda^{\times} \\
          		&  = O(C^{\omega(q)} )\cdot  \max \{  \frac{X}{q^4},  \frac{X^{39/40}}{q^{4-1}}, \frac{X^{38/40}}{q^{4-2}}, \frac{X^{(37+1/3)/40}}{q^{4-3}} ,\frac{X^{(36+ 4/3)/40}}{q^{4-4}} \}\\
          		& = O(C^{\omega(q)} )\cdot  \max \{  \frac{X}{q^4}, \frac{X^{38/40}}{q^{4-2}},\frac{X^{(36+ 4/3)/40}}{q^{4-4}} \}.
          		\end{aligned}
          		\end{equation}
          		
          		We know that there are a lot of regions containing irreducible points for $S_5$ extensions. However notice that the last term above is $X^{(37+1/3)/40}$, therefore we will not compute for those regions with a total counting smaller than this. They must contribute an even smaller counting when we consider this restriction in those regions. By \cite{Bha10} Table $1$, we can see that there are still three left to be considered when $a^1_{12} =0$: \\
          		2. $a^1_{13}\ne0$, $a^2_{12} \ne 0$; \\
          		3. $a^1_{13} = 0$ but $a^1_{14}, a^1_{23}, a^2_{12} \ne 0$; \\
          		4. $a^2_{12} =0$, but $a^1_{13}, a^3_{12} \ne 0$.
          		
          		For $2$, $D_{\lambda} = \{ s_i\ge C_i \mid  s_1^{3}s_2s_3s_4^{3}s_5s_6^{4}s_7^{2}\le \lambda /\kappa,  s_1^{-1}s_2s_3s_4^{3}s_5^{6}s_6^{4}s_7^{2}\le \lambda /\kappa  \}$. The definition of $D_{\lambda}$ makes it clear that for all $t_{lij}\ge t_{113}, t_{212}$ in the partial order we define, we have $t_{lij} \ge \kappa$. And $t_{112}$ could be arbitrarily small. So we will assume $t_{112}$ to be $1$ when we plug into Theorem \ref{mrBoverq} and get an upper bound on $L^2$: 
          		\begin{equation}
          		\begin{aligned}
          		L^2 = & O\left(\frac{\prod_{i=2}^{40} r_i}{q^k}\right) \cdot C^{\omega(q)} \cdot \max\{1, q, \frac{q^2}{r_i} , \dots, \frac{q^k}{ \prod_{i= i_1}^{i_{k-1}} r_i}\}\\ 
          		= &  O(\frac{\lambda^{40}}{q^4}) \cdot C^{\omega(q)} \cdot \max\{ \frac{q}{\lambda t_{112}}, \frac{q^2}{\lambda^2 t_{112} t_{113}}, \frac{q^2}{\lambda^2 t_{112} t_{212}},  \frac{q^3}{\lambda^3 t_{112} t_{113} t_{123}} , \frac{q^3}{\lambda^3 t_{112} t_{113} t_{114}}, \\
          		& \frac{q^3}{\lambda^3 t_{112} t_{113} t_{212}}, \frac{q^3}{\lambda^3 t_{112} t_{212} t_{312}},  \frac{q^4}{\lambda^4 t_{112} t_{113} t_{114}t_{123}}, \frac{q^4}{\lambda^4 t_{112} t_{113} t_{114}t_{212}}, \frac{q^4}{\lambda^4 t_{112} t_{113} t_{123}t_{212}},\\
          		& \frac{q^4}{\lambda^4 t_{112} t_{113}t_{212}t_{213}}, \frac{q^4}{\lambda^4 t_{112} t_{113}t_{212}t_{312}} , \frac{q^4}{\lambda^4 t_{112} t_{212}t_{312}t_{412}} \}. \\
          		\end{aligned}
          		\end{equation}
          		The list $L^2$ contains everything in $L^1$ except the first term $O(\frac{\lambda^{40}}{q^k})\cdot C^{\omega(q)}$. As considered before, we only need to focus on those difficult terms and it suffices to see that $s_1 \le O(\lambda^{1/3})$ again in this $D_{\lambda}$. 
          		
          		For $3$ and $4$, things can be done similarly. In case $3$, $a_{114} \ne 0$ and $a_{123}\ne 0$ together implies that $t^{-1}_{114}t^{-1}_{123} = s_1^{6}s_2^{2}s_3^{2}s_4s_5^{2}s_6^{3}s_7^{4} \le O(\lambda^2)$, so $s_1\le O( \lambda^{1/3} )$. In case $4$, $a_{113} \ne 0$ implies that $s_1^{3}s_2s_3s_4^{3}s_5s_6^{4}s_7^{2} \le O(\lambda)$, so $s_1\le O( \lambda^{1/3} )$. 
          		
          		Therefore, we get the uniformity result for $N_q(S_5, X) = O( \frac{X}{q^{4/15-\epsilon}})$. 
          	\end{proof}  
         In order to prove Theorem \ref{uni5} over arbitrary number field $K$, we will need to prove the analogue of Theorem \ref{mrBoverq} over an arbitrary number field $K$. The setup is a bit more complex than the case over $\mathbb{Q}$. The variety that describes points with extra ramification is defined over $O_K$. Since $\rho:O_K\hookrightarrow \mathbb{R}^{r}\bigoplus \mathbb{C} ^{s}$ is a full lattice, an $O_K$-point on the variety corresponds to a lattice point in $\mathbb{R}^{dn} \simeq (\mathbb{R}^{r}\bigoplus \mathbb{C} ^{s})^n$ where $d$ is the degree of $K/\mathbb{Q}$. Denote $\mathbb{R}^{r}\bigoplus \mathbb{C}^{s}$ by $F$. 
         The scaling vector is $r = (r_1, \dots, r_n)$ where $r_i\in F$ for each $i$. Define $|\cdot|_{\infty}$ to be the norm in $F$: $|v|_{\infty} = \prod_{r} |v_i|_i \prod_{s} |v_j|_j$ where $|\cdot|_i$ means standard norm in $\mathbb{R}$ at real places and square of standard norm in $\mathbb{C}$ at complex places.         
       \begin{theorem}\label{mrBoverK}
       	Let $B$ be a compact region in  $F^n \simeq\mathbb{R}^{nd}$ with finite measure. Let $Y$ be any closed subscheme of $\mathbb{A}^{n}_{O_K}$ of codimension $k$. Let $r = (r_1, \dots, r_n)$ be a diagonal matrix of non-zero elements where $|r_i|_{\infty} \ge \kappa$ for a certain $\kappa$. Let $q$ be a square free prime ideal in $O_K$ and $m$ be a lower triangle unipotent transformation in $GL_n(F)$. Then we have 
       	\begin{equation}
       	\begin{aligned}
       	&\sharp \{  a\in mrB\cap (O_K)^n\mid a (\text{mod }  q)\in Y(O_K/qO_K) \} \\
       	= & O(\frac{\prod_{i=1}^n |r_i|_{\infty}}{|q|^k}) \cdot C^{\omega(q)} \cdot \max\{1, \frac{|q|}{|r_i|_{\infty}}, \frac{|q|^2}{|r_i r_j|_{\infty}} , \dots, \frac{|q|^k}{ \prod_{i= i_1}^{i_k} |r_i|_{\infty}}\}
       	\end{aligned}
       	\end{equation}
       	where the implied constant depends only on $B$, $Y$ and $\kappa$, and $C$ is an absolute constant only depending on $Y$. 
       \end{theorem} 
       In order to prove this analogue, we need the following lemma on the regularity of shapes of the ideal lattices for a fixed number field $K$. Given an integral ideal $I\subset O_K$, we can embed it to $F$ as a full lattice with covolume compared with $O_K$ to be $[O_K : I] = \Nm_{K/\mathbb{Q}}(I)$.    
       \begin{lemma}\label{reglam}
       	Let $K$ be a number field and $I\subset O_K$ be an arbitrary ideal. Given $\lambda= (\lambda_i) \in F = \mathbb{R}^{r}\bigoplus \mathbb{C} ^{s} $, then	
       	$$\sharp \{  a \in I \mid \forall i, |\sigma_i(a)|_i \le |\lambda_i|_i  \} = O(\frac{|\lambda|_{\infty}}{|I|}) + 1 $$
       	where $\sigma_i$ for $i =1, \dots, r+s$ are the Archimedean valuations of $K$ and $|\cdot|_i$ is the usual norm in $\mathbb{R}$ for real embeddings and square of the usual norm in $\mathbb{C}$ for complex embeddings . The implied constant depends only on $K$.
       \end{lemma}
       \begin{proof}
        Given $I$ in the ideal class $R$ in the class group of $K$, denote $[a]$ to be the equivalence class of non-zero $a$ in $I$ where $a\sim a'$ if $a = ua'$ for some unit $u$. Then we have \cite{SL}
       	\begin{equation}
       	\begin{aligned}
       	\sharp \{ [a] \in I \mid  |[a]|_{\infty} \le |I|X \} = \sharp \{\alpha \subset O_K \mid \alpha\in R^{-1}, |\alpha| < X  \} = O(X).
       	\end{aligned}
       	\end{equation}
       	To take advantage of the equality above, we cover the original set $W$ by a disjoint union of subsets $W_k$: 
       	\begin{equation}
       	\begin{aligned}
       	W = \{  a \in I \mid \forall i, |\sigma_i(a)|_i \le |\lambda_i|_i   \} \backslash \{ 0 \}  = \bigcup_{k\ge 1} \{  a \in I \mid \forall i, |\sigma_i(a)|_i \le |\lambda_i|_i  ,  \frac{|\lambda|_{\infty}}{2^k} \le |a|_{\infty} \le \frac{|\lambda|_{\infty}}{2^{k-1}} \} = \cup_k W_k.
       	\end{aligned}
       	\end{equation}
       	For $a\in W_k$, we have that 
       	$$  \frac{|\lambda_i|_i}{2^{k}}\le |\sigma_i(a)|_i \le |\lambda_i|_i,$$ and if $ua$ is in $W$, it must be also in the same $W_k$ since $|ua|_{\infty} = |a|_{\infty}$. So the magnitude of $u$ is bounded as $ 2^{-k}\le |\sigma_i(u)|_i \le 2^{k}$ by the above inequality. By Dirichlet's unit theorem, the units of $K$ aside from roots of unity after taking logarithm form a lattice of rank $r+s-1$ satisfying $\sum_i \ln |\sigma_i(u)|_i = 0$, therefore 
       	$$\sharp \{ u\in O_K^{\times} \mid |\ln |\sigma_i(u)|_i|\le k \}  = O (k^{r+s-1}).$$
       	So for each $[a]\in W_k$, the multiplicity is bounded by $O(k^{r+s-1})$, and the number of equivalence classes in $W_k$ is bounded by
       	\begin{equation}
       	\begin{aligned}
       	\sharp \{ [a] \in I \mid  |a|_{\infty} < \frac{|\lambda|_{\infty}}{2^{k-1}}\} \le O (\frac{|\lambda|_{\infty}}{|I|}\cdot \frac{1}{2^{k-1}}).
       	\end{aligned}
       	\end{equation}
       	Therefore 
       	\begin{equation}
       	\begin{aligned}
       	|W_k| \le O (\frac{|\lambda|_{\infty}}{|I|}) \cdot \frac{k^{r+s-1}}{2^{k-1}}.
       	\end{aligned}
       	\end{equation}
       	The total counting by summation over all $k$ is
       	$$\sharp \{  a \in I \mid \forall i, |\sigma_i(a)|_i \le |\lambda_i|_i  \}\backslash \{0 \} = \sum_k |W_k|\le O(\frac{|\lambda|_{\infty}}{|I|}) \sum_k  \frac{k^{r+s-1}}{2^{k-1}}  \le O(\frac{|\lambda|_{\infty}}{|I|}). $$
       	So the total counting with the origin is
       	$$\sharp \{  a \in I \mid \forall i, |\sigma_i(a)|_i \le |\lambda_i|_i  \} = O(\frac{|\lambda|_{\infty}}{|I|}) + 1.$$
       \end{proof}
       A corollary of this lemma is that the shape of the ideals lattices inside $O_K$ cannot be too skew. We will make this precise in the following lemma and prove it by a more direct approach.               
       \begin{lemma}\label{regsquare}
       Given a number field $K$ with degree $d$, for any integral ideal $I \subset O_K$, denote $\mu_i$ to be the successive minimum for the Minkowski reduced basis for $I$ as a lattice in $\mathbb{R}^d$. Then $\mu_i$ is bounded by
              	$$ \mu_i \le O(|I|^{1/d})$$
       for all $1\le i \le d$. The implied constant only depends on the degree of $K$, the number of complex embeddings of $K$ and the absolute discriminant of $K$. 
       \end{lemma}
      \begin{proof}
              	Given an integral ideal $I$, and an arbitrary non-zero element $\alpha \in I$, we have $(\alpha) \subset I$, so $|(\alpha)| \ge |I|$. The length of $\alpha$ in $\mathbb{R}^{d}$ is 
              	\begin{equation}
              	\begin{aligned}	
              	&\sqrt{|\alpha|_1^2 + \cdots + |\alpha|_r^2 + |\alpha|_{r+1} + \cdots + |\alpha|_{r+s} }\\
              	\ge & \sqrt{ d (\prod_{1\le i\le r} |\alpha_i|^2\prod_{r+1\le i\le r+s} \frac{|\alpha|_i^2}{4})^{1/d}}\\
              	\ge & \sqrt{d} 2^{-s/d} |(\alpha)|^{1/d}\\
              	\ge & \sqrt{d} 2^{-s/d} |I|^{1/d}.
              	\end{aligned}
              	\end{equation}
              	The first inequality comes from the fact that the arithmetic mean is greater than the geometric mean. While Minkowski's first theorem guarantees that $\mu_1 \le O(|I|^{1/d})$, we can bound $\mu_1$ by $O(|I|^{1/d})$ in the other direction. This amounts to saying that the first minimum $\mu_1$ of Minkowski's reduced basis is exactly at the order of the diameter $O(|I|^{1/d})$. Moreover Minkowski's second theorem states that
              	$$\prod_{1\le i\le d} \mu_i \le 2^d D_K^{1/2}|I|,$$
              	therefore for all $i\le d$,
              	$$\mu_i \le O(|I|^{1/d})$$
              	where the implied constant only depends on $d$, $s$ and $D_k$. 
      \end{proof}
      \begin{remark}
      	By Lemma \ref{reglam}, if we pick $\lambda$ with $|\lambda|_{\infty} = O(|I|)$ such that $|\lambda_i|_i = O(|I|^{1/d})$ for real places and $|\lambda_i|_i = O(|I|^{2/d})$ for complex places, we get a square box with side length $O(|I|^{1/d})$ in $\mathbb{R}^d$. Since the first term in Lemma $4.7$ could be bounded by $O(\frac{|\lambda|_{\infty}}{|I|}) = O(1) $, we can find a uniform upper bound of  $C(|I|^{1/d})$ on the side length such that the only lattice point in a smaller square box is the origin. Therefore the first successive minimum $\mu_1$ is greater than the upper bound. 
      	\end{remark}
              On the other hand, the Minkowski's reduced basis generates the whole lattice with covolume $|I|D_K^{1/2}$, so the angle among the vectors in the basis is away from zero. This basically means that among the family of lattices of all integral ideals of $K$ under Minkowski's reduced basis all look like square boxes, and we can find a fundamental domain within the square box. 
              \begin{corollary}\label{regc}
              	Given a number field $K$ with degree $d$, for any integral ideal $I \subset O_K$ and any residue class $c\in O_K/IO_K$, denote $c_i$ to be the $i$-th coordinate in $\mathbb{R}^d$. Then we can find a representative $c$ such that each $$ |c_i| \le O(|I|^{1/d})$$
              	for all $1\le i \le d$. The implied constant depends only on $K$. 
              \end{corollary}      	
        \begin{proof}[\textbf{Proof of Theorem \ref{mrBoverK}}]
        	The case where $k =0$ is trivial since the number of lattice points in the box is $O(\prod_{i=1}^n |r_i|_{\infty})$. It suffices to prove the statement for the initial case when $k = 1$ and $n=1$. The induction procedure works similarly with Theorem \ref{mrBoverq}. 
        	
        	There is only one polynomial $f(x)$ to be considered for $n=1$ and $k=1$. Since $q$ is square free, the number of solution in $O_K/qO_K$ is bounded by $C^{\omega(q)}$ by Chinese remainder theorem. Therefore the solutions of $f(\text{mod }q)$ in $O_K$ is a union of $C^{\omega(q)}$ translations $q+c$ of the lattice $q$ where $c$ is a certain residue class in $O_K/qO_K$ that is also a solution. 
        	
        	Lemma \ref{reglam} states that for arbitrary $r\in F$,
        	$$\sharp \{  a \in rB \cap O_K \mid a\in 0+q  \} = O (\max \{ \frac{|r|_{\infty}}{|q|}, 1 \})$$
        	when $B$ is a unit square in $F$. It follows that the equality is true for any general compact set $B$ since it could be covered by a square and then the implied constant will also depend on $B$. For other nontrivial translations by a root $c$, we have
        	\begin{equation}
        	\begin{aligned}	
        	 \sharp \{  a \in rB \cap O_K \mid a\in c+q  \} 
        	=  \sharp \{  a\in (rB-c) \cap O_K \mid a\in q  \}.
        	\end{aligned}
        	\end{equation}
        	So it is equivalent to consider the number of lattice points in a translation of the box. We could cover $B$ by $2^n$ sub-boxes $B_s$ which is defined by sign in each $\mathbb{R}$ space. Then $rB-c$ could be covered by $rB_s -c$. It suffices to count the lattice points in each $rB_s-c$ and add them up. For each $s$, if there exists one lattice point $P \in rB_s-c$, then we can cover $rB_s-c$ by $P+rB_s$, and the number of lattice points in $rB_s+P$ is equivalent to that in $rB_s$ which is
        	$$\sharp \{  (P+rB_s) \cap q \} = \sharp \{  rB_s\cap q \} \le O (\max \{ \frac{|r|_{\infty}}{|q|}, 1 \}).$$
        	If there are no lattice points in $B_s$, then there is nothing to add. Altogether we have that for any residue class $c$ and any compact set $B$,
        	$$ \sharp \{  a \in rB \cap O_K \mid a\in c+q  \}\le O ( 2^n \max \{ \frac{|r|_{\infty}}{|q|}, 1 \}) =  O ( \max \{ \frac{|r|_{\infty}}{|q|}, 1 \}).$$
        	Here the implied constant depends only on $B$ and $K$. 
        	Adding up all solutions of $f$, we get 
        	$$\sharp \{  a\in rB\cap O_K \mid f(a) \equiv 0 \text{ mod } q \} = O(\frac{ |r|_{\infty}}{|q|}) \cdot C^{\omega(q)} \cdot \max\{1, \frac{|q|}{|r|_{\infty}}\}.$$        	
        	This finishes the proof for the case $k = 1$, $n=1$. 
        \end{proof}
       	Finally, based on Theorem \ref{mrBoverK}, we can prove Theorem \ref{uni5} over a number field $K$. 
       	\begin{proof}[\textbf{Proof of Theorem \ref{uni5} over $K$}]
       		We will follow the notation \cite{BSW15} in this proof. Counting $S_n$-number fields for $n= 3,4,5$ over a number field $K$ is different from that over $\mathbb{Q}$ mostly in two aspects. 
       		
       		Firstly, the structure of finitely generated $O_K$-module is more complex than that of $\mathbb{Z}$, therefore the parametrization of $S_n$ number fields over $K$ will involve other orbits aside from $G(O_K)$-orbits of $V(O_K)$ points. Actually finitely generated $O_K$-modules with rank $n$ are classified in correspondence to the ideal class group $\text{Cl}(K)$ of $K$. So for each ideal class $\beta$, we get a lattice $\mathcal{L}_\beta$ corresponding to $S_n$ extensions $L$ with $O_L$ corresponding to $\beta$. We just need to count the number of orbits in $\mathcal{L}_\beta$ under the action of $\Gamma_\beta$ where $\Gamma_\beta$ is commensurable with $G(O_K)$ and $\mathcal{L}_\beta$ is commensurable with $V(O_K)$. See section $3$ in \cite{BSW15} for more details. 
       		
       		Secondly, the reduction theory over a number field $K$ is slightly different in that the description of fundamental domain requires the introduction of units, and this effect of units is especially beneficial for summation over fundamental domain. The most significant difference is at the description of the torus. Originally over $\mathbb{Q}$, we have $G(\mathbb{R})/ G(\mathbb{Z})  = NA\mathcal{K}\Lambda$ \cite{Bha10} where $A$ is an $l$-dimensional torus ($l=7$ for $S_5$) embedded into $\text{GL}_n(\mathbb{R})$ ($n=40$ for $S_5$) as diagonal elements
       			$$ T(c) = \{ t(s_1,\dots, s_l)\in T(\mathbb{R}) = \mathbb{G}^l_m(\mathbb{R})\mid \forall i, s_i\ge c\}.$$
       		Given a number field $K$, recall that $\rho:O_K\hookrightarrow F = \mathbb{R}^{r}\bigoplus \mathbb{C}^{s}$ is the embedding of $O_K$ as a full lattice in $\mathbb{R}^d$. Then $A$ could be described as a subset of 
       		$$ T(c,c') = \{ t = t(s_1,\dots, s_l)\in T(F) = \mathbb{G}^l_m(F)\mid \forall i, |s_i|_{\infty} \ge c , \forall j,k, \ln \frac{|s_i|_{j}}{|s_i|_{k}} \le c' \}.$$
       		Here $|s_i|_j \le O(|s_i|_{k})$ for all $j, k$ guarantees that $|s_i|_j\sim |s_i|_{k}$, thus $|s_i|_v\sim |s_i|_{\infty}^{1/(r+s)}$. Therefore, if we have a bound that $|s_i|_{\infty} \le A$, then we can get the bound $ |s_i|_v\le O(A^{1/r})$. See section $4$ \cite{BSW15} for more details. 
       		
            Recall that we need to compute
       		\begin{equation}
       		\begin{aligned}
       		N(S; X) = \frac{1}{M_i}\int_{g\in \mathcal{F}} \sharp \{x \in S^{irr} \cap gB\cap V^{(i)}_{F}: |\text{Disc}(x)|_{\infty}<X \}  dg
       		\end{aligned}
       		\end{equation}
       		where $V_F^{(i)}$ is a subspace of $V_F$ with a certain signature, and $B$ is a compact ball in the space $V_F$ that is invariant under the action of the orthogonal group $K$. By Theorem \ref{mrBoverK}, the integrand is 
       		\begin{equation}
       		\begin{aligned} 
       		& \sharp \{x \in S^{irr} \cap gB\cap V^{(i)}_{F}: |\text{Disc}(x)|_{\infty}<X \} 
       		\le \sharp \{x \in m\lambda t B\cap \mathcal{L} \mid  x (\text{mod } q) \in Y(\mathbb{Z}/ q\mathbb{Z})\}\\
       		= & O(\frac{|\lambda|_{\infty}^n}{|q|^k}) \cdot C^{\omega(q)} \cdot \max\{1, \frac{|q|}{|\lambda t_i|_{\infty}}, \frac{|q|^2}{|\lambda ^2t_it_j|_{\infty}} , \dots, \frac{|q|^k}{|\lambda ^k \prod_{i= i_1}^{i_k} t_i|_{\infty}}\}.
       		\end{aligned}
       		\end{equation}
       		Here in order to present the result in a similar form with that over $\mathbb{Q}$, for each $\lambda\in \mathbb{R}^{+}$ we denote $\lambda$ to be the diagonal matrix such that $|\text{Disc}(\lambda v)|_{\infty} = |\lambda|^n_{\infty} |\text{Disc}(v)|_{\infty}$ where $n= 40$ for $S_5$. 
            
            The first case is to compute $G(O_K)$-orbits in $V(O_K)$, which corresponds to the trivial class in $\text{Cl}(K)$. Denote $\mathcal{F}$ to be  $G(F)/G(O_K)$ and $\mathcal{L}$ to be the image of $V(O_K)$ in $V(F)$. We first look at the case where $a^{1}_{12}\ne 0$. Since $\mathcal{L}$ is a lattice, $x$ with non-zero $a^1_{12}$ is away from zero and $|a|_{\infty}$ could be bounded from below by $\kappa$, so we would only integrate over $$D_\lambda = \{ t= t(s_i)\in T(c,c')\mid |s_1^{3}s_2s_3s_4^{3}s_5^{6}s_6^{4}s_7^{2}|_{\infty}\le \lambda /\kappa\}.$$ The integral over $F= \mathbb{R}^d$ gives the same result as over $\mathbb{Q}$
            \begin{equation}
            \begin{aligned}
            &\int_{O(1)}^{A} |s|^u_{\infty}\text{d}s^{\times} \le \prod_{1\le i\le r} \int_{O(1)}^{O(A^{1/(r+s)})} s_i^{u}\text{d}s_i^{\times} \prod_{r+1\le i\le r+s}
         \int_{O(1)}^{O(A^{1/2(r+s)})} r_i^{2(u-1)}r_i\text{d}r_i = O(A^{u}).
            \end{aligned}
            \end{equation}
            So we will end up with the same result over $K$.

       		For fields corresponding to other ideal class $\beta \in \text{Cl}(K)$, we can similarly compute the average number of lattice points in $\mathcal{F}v$ for $v\in B$ with bounded discriminant. Denote $\mathcal{F}_\beta = \Gamma_\beta\backslash G(F)$. By \cite{BSW15}, we can cover $\mathcal{F}_\beta$ by finitely many $g_i\mathcal{F}$ where $g_i \in G(O_K)$ are representatives of $G(O_K)/ (G(O_K)\cap \Gamma_\beta)$. Let's call $\mathcal{D}_i = \mathcal{F}_\beta \cap g_i\mathcal{F}$, then we just need to sum up
       		\begin{equation}\label{sumbeta}
       		\begin{aligned}
       		&\frac{1}{M_i}\int_{g\in \mathcal{D}_i} \sharp \{x \in S^{irr} \cap gB\cap V^{(i)}_{F}: |\text{Disc}(x)|_{\infty}<X \}  dg\\
       		\le  &   \frac{1}{M_i}\int_{g\in g_i\mathcal{F}} \sharp \{x \in S^{irr} \cap gB\cap V^{(i)}_{F}: |\text{Disc}(x)|_{\infty}<X \}  dg\\
       		\le &  \frac{1}{M_i}\int_{g\in \mathcal{F}} \sharp \{x \in g_i^{-1}S^{irr} \cap gB\cap V^{(i)}_{F} \}  dg.
       		\end{aligned}
       		\end{equation}  
       		As in \cite{BSW15} section $3$,
       		$$ \mathcal{L}_{\beta}:= V_n(K) \cap \beta^{-1} \prod_{p\nmid \infty} V(O_p) \prod_{p|\infty} V(F_p)$$
       		where $\beta$ is a representative of the double coset $\text{cl}_S = ( \prod_{p\nmid \infty} G(O_p))\backslash G(\mathbb{A}_f) /G(K)$. Here $\mathbb{A}_f $ is the restricted product of $K_p^{\times}$ for all finite places $p$. So given a class $\beta$, we can choose a representative such that $\beta_p$ is the identity element in $G(O_p)$ except at a finite set of places $S$. At $p \in S$,  $\beta_p$ is in $G(K_p)$. Given $v\in \mathcal{L}_{\beta}$, we have 
       		$$v_p\in \beta_p^{-1} V(O_p).$$
       		Since $\beta_p^{-1}$ can be regarded as a linear action, there must exist $r$ large enough such that 
       		$$v_p \pi^r \in \beta_p^{-1} \pi^r V(O_p) \in V(O_p)$$
       		and $(\pi^r) = (a_p)$ is a principle integral ideal in $O_K$ where $\pi$ is a uniformizer for $O_p$. Glue all the $a_p$ and we get $a = \prod_{p\in S} a_p$. By the way it is defined, we have that 
       		$a\mathcal{L}_\beta$ is in $O_K$ and $a\in O_p^{\times }$ at $p\notin S$. So for $p$ outside $S$, $v\in \mathcal{L}_\beta$ is in $Y(O_K/p)$, if and only if, $av \in O_K$ is in $Y(O_K/p)$. 
            Therefore we can consider $a\mathcal{L}_{\beta}$ inside $O_K$ instead and do not lose the information of ramification at all but finitely many places. Since there are only finitely many ideal classes it will not affect the form of the uniformity estimate but only the implied constant. From now on, we will assume $\mathcal{L}_{\beta}$ to be in $O_K$. 
       		
       		In (\ref{sumbeta}), $S^{irr}$ denotes the set of totally ramified points at $q$ in $\mathcal{L}_{\beta}$. If $q$ is a square free integral ideal away from $S$ and $x \in S^{irr}$ satisfies $x\in O_K$ and $x\in Y(O_K/ q)$, then $g_i^{-1} v \in O_K$ and $ g_i^{-1} v \in g_i^{-1} Y(O_K/q)$. Denoting $g_i^{-1} Y = Y_i$, then it suffices to count 
       		\begin{equation}
       		\begin{aligned}
       		\sharp \{x \in g_i^{-1}\mathcal{L}_{\beta}\cap gB \cap Y_i(O_K/q) \}.
       		\end{aligned}
       		\end{equation}  
       		Since $g_i^{-1}Y$ only differs with $Y$ by a linear transformation on coordinates, $Y_i$ has the same codimension. Apply Theorem \ref{mrBoverK} to get the same estimates. To consider arbitrary square free ideal $q = q_1q_2$ with $q_2$ containing the involved factors in $S$, we can estimate with $q_1$ and replace $|q_1|$ by $|q|$ with a difference of at most $O(1)$ since there are only finitely many $p\in S$. 	
       	\end{proof}

       \subsection{Local uniformity for Abelian extensions}
       It has been proved \cite{Wri89} that Malle's conjecture is true for all abelian groups over any number field $k$.
       \begin{theorem}
	   Let $A$ be a finite abelian group and $k$ be a number field, the number of $A$-extensions over $k$ with the absolute discriminant bounded by $X$ is
	                 $$ N(A,X) \sim C X^{1/a(A)}(\ln X)^{b(k,A)-1}.$$
       \end{theorem}
       We will need to prove a uniformity estimate for $A$ extensions with certain local conditions. For an arbitrary integral ideal $q$ in $O_k$, define $N_q(A,X) = \sharp\{ K\mid \Disc(K/k)\le X, \Gal(K/k) = A, q| \disc(K/k) \}$. 
       \begin{theorem}
  	    Let $A$ be a finite abelian group and $k$ be a number field, then 
	        $$ N_q(A,X) \le O(C^{\omega(q)})(\frac{X}{|q|})^{1/a(A)} (\ln X)^{b(k, A)-1}$$
	    for an arbitrary integral ideal $q$ in $O_k$, where $C$ and the implied constant depends only on $k$ .
      \end{theorem}
      \begin{proof}
      	We will follow the notation and the language of \cite{Woo10a} to describe abelian extensions. To get an upper bound of $A$-number fields, it suffices to bound on the number of continuous homomorphisms from the id\`ele class group $C_k\to A$. Similarly, for $A$-number fields with certain local conditions, it suffices to bound on the number of continuous homomorphisms from the id\`ele class group $C_k\to A$ satisfying certain local conditions. 
      	
      	
      	Let $S$ be a finite set of primes such that $S$ generates the class group of $k$, including infinite primes and possibly wildly ramified primes, i.e., primes above the prime divisors of $|A|$. Denote $J_k$ to be the id\`ele group of $k$, $J_S$ to be the id\`ele group with component $O_v^{\times}$ for all $v\notin S$ and $O_S^*$ to be $k^*\cap J_S$. By lemma 2.8 in \cite{Woo10a}, the id\`ele class group $ C_k = J_k/k^{\times}\simeq J_S/O_S^{\times} $. Therefore to bound the number of continuous homomorphisms $C_k\to A$, we can choose to bound the number of continuous homomorphisms $J_S\to A$. The Dirichlet series for $J_S\to A$ with respect to absolute discriminant is an Euler product, see \cite{Woo10a} section 2.4, 
      	\begin{equation}
      	\begin{aligned}	
      	F_{S, A} (s) & =\prod_{p\in S} (\sum_{\rho_p: k_p^*\to A} |p|^{-d(\rho_p)s})   \prod_{p\notin S}(\sum_{\rho_p: O_p^*\to A} |p|^{-d(\rho_p)s}) = \sum_n \frac{a_n}{n^{s}}
      	\end{aligned}
      	\end{equation}
      where $d(\rho_p)$ is the exponent of $p$ in the relative discriminant and can be determined by the tame inertia group at $p$, which is the image of $O_p^*$ in $A$. Lemma 2.10 \cite{Woo10a} shows that $F_{S,A}(s)$ has exactly the same right most pole with Dirichlet series for $A$-number fields at $s = \frac{1}{a(A)}$ with the same order $b(k,A)$.
      	
      	$F_{S,A}(s)$ is a nice Euler product: for all $p$-factor there is a uniform bound $M$ on the magnitude of coefficient $a_{p^r}$ and a uniform bound $R$ on $r$ such that $a_{p^r}$ is zero for $r>R$. Denote the counting function of $F_{S,A}(s)$ by $B(X) = \sum_{n \le X} a_n$. Then for a certain integer $q = \prod_i p_i^{r_i}$, denote $B_q(X) = \sum_{q|n<X} a_n$. Let $q_0 = \prod_i p_i^{R}$ then
      	\begin{equation}
      	\begin{aligned}	
         B_q(X) & = \sum_{q|d|q_0}  a_d \sum_{k, (d,k)= 1, dk<X} a_k \le  \sum_{q|d|q_0}  a_d  B(\frac{X}{d})
          \le  \sum_{q|d|q_0}  M^{\omega(q)} (\frac{X}{d})^{1/a(A)} \ln^{b(A)-1}X\\
         & = M^{\omega(q)} X^{1/a(A)}\ln ^{b(A)-1}X  \sum_{q|d|q_0} \frac{1}{d^{1/a(A)}}\\
         & \le (MR)^{\omega(q)} X^{1/a(A)}\ln^{b(A)-1} X \frac{1}{q^{1/a(A)}} = O(C^{\omega(q)})(\frac{X}{q})^{1/a(A)} \ln^{b(A)-1} X.
      	\end{aligned}
      	\end{equation}	
       We have $N_q(A,X)$ bounded by $B_{|q|}(X)$ for an arbitrary integral ideal $q$.
      \end{proof}

       \section{Proof of the Main Theorem}
        In this section, we prove our main results Theorem \ref{Thm1}.
       \begin{lemma}\label{unin}
       	For $n=3,4,5$, let $A$ be an abelian group satisfying the corresponding condition on $m = |A|$ in Theorem \ref{Thm1}. Then $\forall c\in A$ and $k\in S_n$ , 
       	 \begin{equation}
       	 \begin{aligned}	
       	 \ind(k, c) /m -\ind(k) + r_k \ge 1 
       	 \end{aligned}
       	 \end{equation}
       	 where the uniformity $O(X/|q|^{r_k})$ holds for $S_n$ degree $n$ extensions with $k$ as the inertia group at $p|q$. 
       \end{lemma}
       \begin{proof}
       	This can be checked by Lemma \ref{delta3}, \ref{delta4} and \ref{delta5} with Theorem \ref{uni3}, \ref{uni4} and \ref{uni5}.
       \end{proof}       
       Then we are going to prove the main results.     
       \begin{proof}[\textbf{Proof of Theorem \ref{Thm1}}]
    We will describe $S_n\times A$ number fields by pairs of $S_n$ degree $n$ field $K$ and $A$-number fields $L$ $$N(S_n\times A, X)  = \sharp \{ (K,L)| \Gal(K/k) \simeq S_n, \Gal(L/k) \simeq A, \Disc(KL)<X\}.$$ We will write $N(X)$ for short and omit the conditions $\Gal(K/k) \simeq S_n$ and $\Gal(L/k) \simeq A$ when there is no confusion. The equality holds since $S_n$ and odd abelian group have no isomorphic quotient. We will prove this result by three steps.\\
      
    1. \textbf{Estimate pairs by $\Disc(O_K O_L)$}. \\ 
    By Theorem \ref{dpro}, we can get a lower bound for $N(S_n\times A, X)$ by counting the number of pairs by $\Disc(O_KO_L)$. Denote $|A| = m$,  
    \begin{equation}
    \begin{aligned}	
     &N(S_n\times A, X) \\
     & \ge \sharp \{ (K,L)| \Gal(K/k) \simeq S_n, \Gal(L/k) \simeq A, \Disc(O_KO_L) = \Disc(K)^m \Disc(L)^n<X\}.
    \end{aligned}
    \end{equation}
    By Lemma \ref{proneq}, there exists $C_0$ such that $N(S_n\times A, X) \ge C_0 X^{1/m}$ asymptotically. We can get a better understanding of the constant $C_0$ in view of Dirichlet series. Let $f(s)$ be the Dirichlet series of $S_n$ cubic number fields, and $g(s)$ be the Dirichlet series of $A$-number fields. Then the Dirichlet series for $\{(K,L)\}$ with respect to $\Disc(K)^m\Disc(L)^n$ is $f(ms) g(ns)$. The analytic continuation and pole behavior of $f$ and $g$ are both well studied \cite{TT13,Wri89,Woo10a}. It has been shown that $f(s)$ has the right most pole at $s = \frac{1}{\ind(S_n)}= 1$ and $g(s)$ has the right most pole at $s = \frac{1}{\ind(A)}$. Recall that for $A$ arbitrary abelian group, $\frac{m}{\ind(A)} = \frac{p}{p-1}$ where $p$ is the minimal prime divisor of $|A|$, so $\frac{1}{m} > \frac{1}{n\ind(A)}$. Therefore the right most pole of $f(ms)g(ns)$ is at $s=\frac{1}{m}$, and the order of the pole is exactly the order of the pole of $f(s)$ at $s=1$, which is $1$. By Tauberian Theorem\cite{Nar83},
     \begin{equation}
     \begin{aligned}	
     \liminf_{X\to \infty} \frac{N(S_n\times A, X)}{X^{1/m}} \ge \text{Res}_{s=1}f\cdot g(\frac{n}{\ind(S_n)\cdot m}) = \text{Res}_{s=1}f\cdot g(\frac{n}{m}).
     \end{aligned}
     \end{equation}\\
     2. \textbf{Estimate pairs by $\Disc_Y(KL)$}.\\     
     Define $\Disc_Y$ to approximate $\Disc$ as follows:
     	\begin{equation}
     	\Disc_Y(KL) = 
     	\left\{
     	\begin{aligned}
     	\Disc_p(KL)                   &          & |p|\le Y  \\
     	\Disc_p(K)^m \Disc_p(L)^n &          & |p|>Y. \\
     	\end{aligned}
     	\right.
     	\end{equation}
    Recall that $\Disc_p$ means the norm of $p$-factor in the discriminant, while $\Disc_Y$, as described above, is an approximation of $\Disc$. The notation would be distinguished by whether the lower index is capital or little letter. 
    
    Define $N_Y(X) = \sharp\{(K,L)| \Disc_Y(KL)<X \}$. Since $\Disc_Y(KL) \ge \Disc(KL)$, as $Y$ gets larger, we get $N_Y(X) \le N(X)$ which is an increasingly better lower bound for $N(X)$. 
    
    To compute $N_Y(X)$, denote the set of primes smaller than $Y$ to be $\{p_i\}$ with $i=1,\cdots, n$. Let $\Sigma_1$ be a set containing a local \'etale extension over $k_{p_i}$ of degree $n$ for each $|p_i|<Y$ and $\Sigma = (\Sigma_1, \Sigma_2)$ contains a pair of local \'etale extension for each $p_i$. There are finitely many local \'etale extensions of degree $n$ and $m$, so there are finitely many different $\Sigma_i$'s and thus finitely many $\Sigma$'s for a certain $Y$. We will write $K\in \Sigma_1$ if for all $|p|\le Y$ $K_p$ as a local \'etale extension is in $\Sigma_1$. 
    
    For each $\Sigma_1$, we know counting result of $S_n$ cubic field \cite{BSW15} with finitely many local conditions $$N_{\Sigma_1}(S_n, X) = \sharp \{ K| \Gal(K/k) \simeq S_n, K\in \Sigma_1\}$$ and similarly for abelian extensions with in $\Sigma_2$\cite{M85,Wri89,Woo10a}. 
    
    
    
      We can relate $\Disc_Y(KL)$ and $\Disc(KL)$ for pairs $(K,L)\in \Sigma $,
     \begin{equation}
     \begin{aligned}	
     \Disc_Y(KL) & = \prod_{|p|\le Y} \Disc_p(KL) \prod_{|p|>Y} \Disc_p(K)^m \Disc_p(L)^n\\
     & = \Disc(K)^m \Disc(L)^n \prod_{|p|\le Y} \Disc_p(KL) \Disc_p(K)^{-m}\Disc_p(L)^{-n}\\
     & = \frac{\Disc(K)^m \Disc(L)^n }{d_\Sigma}
     \end{aligned}
     \end{equation}
     where $d_{\Sigma}$ is a factor only depending on $\Sigma$. We have seen in section 2 that at tamely ramified primes, $\Disc_p(KL)$ can be determined by inertia groups of $\tilde{K}$ and $\tilde{L}$, therefore it depends on $\Sigma$ at $p$. For wildly ramified primes, it suffices to see that $\Disc_p(KL)$ could be determined by $K_p$ and $L_p$. This is always true under taking product: if $\tilde{K}$ and $\tilde{L}$ have trivial intersection, we can get the map from absolute local Galois group $G_{k_p}$ to $S_n\times A$ by taking the product of such maps to $S_n$ and $A$. Then we get the precise local information for $KL$ including $\Disc_p(KL)$.
          
     Therefore $\Disc_Y(KL)\le X$ is equivalent to $\Disc(K)^m \Disc(L)^n\le d_\Sigma X$ for $(K,L)\in \Sigma$. Apply Lemma \ref{proneq} to $N_{\Sigma_1}(S_n, X)$ and $N_{\Sigma_2}(A, X)$, we get
     \begin{equation}\label{lbound}
        \begin{aligned}	
        \lim_{X\to\infty} \frac{N_Y(X)}{X^{1/m}} = C_Y.
        \end{aligned}
     \end{equation}    
    For each $Y$, $N_Y(X) \le N(X)$, therefore 
    \begin{equation}
    \begin{aligned}	
    \lim_{Y\to \infty} \lim_{X\to\infty} \frac{N_Y(X)}{X^{1/m}} = \lim_{Y\to\infty} C_Y \le \liminf_{X \to \infty} \frac{N(X)}{X^{1/m}}.
    \end{aligned}
    \end{equation}
    By definition of $N_Y$, $C_Y$ is monotonically increasing as $Y$ increases and will be shown to be uniformly bounded in next step. So this limit does exist and gives a lower bound.
    
    
    
    \textbf{3. Bound $N(X) - N_Y(X)$}\\   
    Our goal is to prove the other direction of the inequality \ref{lbound}.
     \begin{equation}
     \begin{aligned}	
     \lim_{Y\to\infty} C_Y \ge \limsup_{X \to \infty} \frac{N(X)}{X^{1/m}},
     \end{aligned}
     \end{equation}
     and thus
    \begin{equation}
    \begin{aligned}	
    \lim_{X\to\infty} \frac{N(X)}{X^{1/m}} = \lim_{Y\to \infty} \lim_{X\to \infty}\frac{N_Y(X)}{X^{1/m}}  = \lim_{Y\to\infty}C_Y .
    \end{aligned}
    \end{equation}
    To get an upper bound of $N(X)$ via $N_Y(X)$, we need to bound on $N(X)-N_Y(X)$. It suffices to show the difference is $o(X^{1/m})$. There are only finitely many wildly ramified primes, so they would only affect the constant but not the order.
    \begin{equation}
    \begin{aligned}	
    N(X) - N_Y(X) & = \sharp \{(K,L)| \Disc(KL) < X < \Disc_Y(KL)\}\\
    & = \sum_{\Sigma'} \sharp\{ (K,L)\in \Sigma'|\Disc(KL) < X < \Disc_Y(KL) \}\\
    \end{aligned}
    \end{equation} 
    where the local condition $\Sigma'$ is a little bit different from $\Sigma$ in last part. Each $\Sigma'$ specifies a finite set of primes $S = \{p_j\}$ and a pair of inertia groups at tame $p$ and a pair of ramified local \'etale extensions at wildly ramified $p$ for each $p$ in S. Denote the pair of local information by $(h_j, g_j)$ for each $p_j$. We will not write the index $j$ each time when there is no confusion. We write $(K,L)\in \Sigma '$ if $K_{p}$ and $L_{p}$ are in $\Sigma'$ for each $p\in S$, and are not ramified simultaneously outside $S$. Denote $\expo(\cdot)$ to be the corresponding exponent of $p$ in discriminant. At tame place, $\expo(\cdot) $ is equal to $\ind(\cdot)$ as described before. For $(K,L) \in \Sigma'$, we can relate precise $\Disc(KL)$ to the product, 
    \begin{equation}
    \begin{aligned}	
    \Disc(KL) = &\Disc(K)^m \Disc(L)^n \prod_{p\in S} |p|^{\expo(h_j, g_j) - m\cdot \expo(h_j)-n\cdot \expo(g_j)}\\
     = &\frac{\Disc(K)^m \Disc(L)^n }{d_{\Sigma'}}.
    \end{aligned}
    \end{equation}
    Each $\Sigma'$ summand is
    \begin{equation}\label{tildepro}
    \begin{aligned}	
    & \sharp\{ (K,L)\in \Sigma'|\Disc(KL) < X < \Disc_Y(KL) \}\\
     \le &\sharp\{ (K,L)\in \Sigma'|\Disc(KL) < X\}\\
     =  &\sharp\{ (K,L)\in \Sigma'|\Disc(K)^m \Disc(L)^n < Xd_{\Sigma'}\}\\
     = &\sharp\{ (K,L)\in \Sigma'|\prod_{p\notin S} \Disc_p(K)^m \Disc_p(L)^n< \frac{X}{\prod_{p\in S} |p|^{\expo(h_j,g_j)} }\}. \\
    \end{aligned}
    \end{equation} 
    Notice only  $\Sigma'$ summand where $\prod_{p\in S} |p|>Y$ is non-zero. Denote $\prod_{p\notin S}\Disc_p(K)$ by $\Disc_{res}(K)$. For a certain $\Sigma'$, define $q_k = \prod_{p\in S, I_p = <k>}' p$ where $\prod'$ means the product is taken only over tamely ramified $p$ in $\Sigma'$. Similarly we write $K\in \Sigma'$ if $K$ satisfies the local conditions specified at $S$ in $\Sigma'$. Then we can bound the number of $K\in \Sigma'$  
    \begin{equation}
    \begin{aligned}	
       & \sharp \{ K | K\in \Sigma', \Disc_{res}(K) \le X\} \\
    = & \sharp \{ K | K \in \Sigma' , \Disc(K)\le X \prod_{p\in S} |p|^{\expo(h_j)}\}\\
    = & O_{\epsilon}\left(\prod_{k}|q_k|^{-r_k} \prod_{p\in S} |p|^{\expo(h_j)}\right) X \\
    = & O_{\epsilon}\left (\prod_{k}|q_k|^{-r_k+\ind(k)}\right ) X.
    \end{aligned}
    \end{equation} 
    Here we can ignore wildly ramified primes since there are only finitely many wildly ramified primes and finitely many wildly ramified local \'etale extensions. Hence the discriminant at those primes are uniformly bounded by some constant. Similarly, 
    \begin{equation}
    \begin{aligned}	
    & \sharp \{ L | L\in \Sigma', \Disc_{res}(L) \le X\} \\
    =& \sharp \{ L | L \in \Sigma' , \Disc(L)\le X \prod_{p\in S} |p|^{\expo(g_j)}\}\\
    =& O_{\epsilon}\left ( (\prod_{p\in S} |p|^{\expo(g_j)})^{\epsilon}\right ) X^{1/a(A)}\ln^{b(A)}.\\
    \end{aligned}
    \end{equation}     
    Now apply Lemma \ref{proneq} to (\ref{tildepro}),
    \begin{equation}\label{meow1}
    \begin{aligned}	
    & \sharp\{ (K,L)\in \Sigma'|\Disc_{res}(K)^m \Disc_{res}(L)^n< \frac{X}{\prod_{p\in S} |p|^{\expo(h_j,g_j)} }\}\\
    \le & O_{\epsilon} \left (\prod_{k}|q_k|^{-r_k+\ind(k) +\epsilon}\right ) (\frac{X}{\prod_{p\in S} |p|^{\expo(h_j,g_j)}})^{1/m}\\
    \le & O_{\epsilon}  \left (\prod_{k}|q_k|^{-r_k+\ind(k) +\epsilon - \ind(k, g_j)/m}\right )X^{1/m}.
    \end{aligned}
    \end{equation} 
    Each $\Sigma'$ gives a list of $(q_k)$ of relatively prime ideals. Conversely, for each list $(q_k)$, there are at most $M^{\omega(\prod_k q_k)} = O_{\epsilon}(\prod_k q_k)^\epsilon $ many $\Sigma'$s, where $M$ is an upper bound of the number of possible tame inertia groups for $A$-extensions, then
      \begin{equation}
      \begin{aligned}	
      N(X) - N_Y(X) & \le \sum_{\Sigma'} \sharp\{ (K,L)\in \Sigma'|\Disc_{res}(K)^m \Disc_{res}(L)^n \le \frac{X}{\prod_{p\in S}|p|^{\expo(h_i, g_i)}}\}\\
      & \le X^{1/m} O_{\epsilon} \left (\sum_{(q_k), \prod_k |q_k|> Y}  \prod_k |q_k|^{\delta} \right )\\
      & \le X^{1/m} O_{\epsilon} (\sum_{|q|>Y} |q|^{\delta+\epsilon})
      \end{aligned}
      \end{equation} 
     where the for every $k$, the exponent $\delta$ is strictly smaller than $-1$ by Lemma \ref{unin} and $\epsilon$ is arbitrary small. Therefore the summation is convergent and $N(X) - N_Y(X)$ is $O(X^{1/m})$ which proves the boundedness of $C_Y$. Moreover,
     \begin{equation}
     \begin{aligned}	
     \lim_{Y\to \infty}\limsup_{X\to\infty}\frac{N(X) - N_Y(X)}{X^{1/m}}  
     \le \lim_{Y\to \infty} \sum_{|q|>Y} O_{\epsilon} (|q|^{\delta+\epsilon} )= 0,
     \end{aligned}
     \end{equation}
    therefore it proves that  
     \begin{equation}
     \begin{aligned}	
     & \limsup_{X\to\infty} \frac{N(X)}{X^{1/m}} \le \lim_{Y\to \infty}\left( \lim_{X\to\infty} \frac{N_Y(X)}{X^{1/m}} + \limsup_{X\to\infty}\frac{N(X) - N_Y(X)}{X^{1/m}}\right)\\
     = &\lim_{Y\to \infty}C_Y.\\
     \end{aligned}
     \end{equation}
    \end{proof}

\section{Acknowledgement}
I am extremely grateful to my advisor Melanie Matchett Wood for constant encouragement and many helpful discussions. I would like to thank Manjul Bhargava, J\"urgen Kl\"uners, Arul Shankar, Takashi Taniguchi, Frank Thorne and Jacob Tsimerman for helpful conversations. I would like to thank, in particular, Manjul Bhargava for a suggestion to improve the uniformity estimate, and Frank Thorne for a suggestion to improve the product lemma. I would also like to thank Manjul Bhargava, Evan Dummit, Gunter Malle, Arul Shankar, Takashi Taniguchi, Takehiko Yasuda for suggestions on an earlier draft. This work is partially supported by National Science Foundation grant DMS-$1301690$.

%
%

\Addresses
\end{document}